\documentclass[11pt]{amsart}
\usepackage{amsfonts, latexsym, amsmath,amstext,amsthm,amssymb,amsxtra}
\usepackage{verbatim,hyperref}
\usepackage{enumitem}
\usepackage[margin=1.2in]{geometry}

\newtheorem{thm}{Theorem}
\newtheorem{thma}{Theorem}

\newtheorem{clm}{Claim}

\newtheorem{innercustomclm}{Claim}
\newenvironment{customclm}[1]
  {\renewcommand\theinnercustomclm{#1}\innercustomclm}
  {\endinnercustomclm}

\newtheorem{lem}{Lemma}[section]
\newtheorem{rmk}{Remark}[section]
\newtheorem{prop}{Proposition}

\newtheorem{obs}{Observation}[section]
\newtheorem*{obs*}{Observation}

\newtheorem*{cor*}{Corollary}

\newcommand{\ds}{\displaystyle}
\newcommand{\norm}[1]{\left\Vert#1\right\Vert}

\newcommand{\ptwo}[2]{\frac{\partial^{2}}{\partial #1\ \partial #2}}
\newcommand{\pone}[1]{\frac{\partial}{\partial #1}}

\newcommand{\sinc}{\mathrm{sinc}}
\newcommand{\flip}[1]{\mathrm{flip} \{#1\}}

\newcommand{\ov}{\overline}

\newcommand{\FF}{{\mathcal F}}
\newcommand{\R}{\mathbb R}
\newcommand{\C}{\mathbb C}
\newcommand{\E}{\mathbb E}
\newcommand{\N}{\mathbb N}
\newcommand{\Pro}{\mathbb P}
\newcommand{\lm}{\lambda}

\newcommand{\g}{\gamma}
\newcommand{\al}{\alpha}
\newcommand{\ep}{\varepsilon}
\newcommand{\p}{\varphi}

\newcommand{\var}{\mathrm{var}\,}
\newcommand{\cov}{\mathrm{cov}\,}
\newcommand{\intt}{\int_{-T}^T}
\newcommand{\ind}{1{\hskip -2.5 pt}\hbox{I}}
\newcommand{\V}{V^{a,b}_f(T)}
\newcommand{\D}{\Delta}

\newcommand{\diag}{\mathrm{Diag}_\ep}
\newcommand{\indep}{\ind\{(\lm,\tau):\, |\lm-\tau|<\ep\}}

\newcommand{\calN}{\mathcal{N}}
\newcommand{\M}{\mathcal{M}(\R)}
\newcommand{\Mplus}{\mathcal{M}^{+}(\R)}

\title[Variance of Zeroes of Shift-Invariant GAFs]{Variance of the Number of Zeroes of Shift-Invariant Gaussian Analytic Functions}

\author[N. D. Feldheim]{Naomi Dvora Feldheim}
\thanks{Department of Mathematics, Stanford University.
Email: naomifel@stanford.edu. \\
Research supported by the Science Foundation of the Israel Academy
of Sciences and Humanities, grant 166/11; by the United States
- Israel Binational Science Foundation, grant 2012037; and by a National Science Foundation postdoctoral fellowship grant.
}

\begin{document}
 \maketitle
\begin{abstract}
Following Wiener, we consider the zeroes of Gaussian analytic functions in a
strip in the complex plane, with translation-invariant distribution. We show
that the variance of the number of zeroes in a long horizontal rectangle
$[-T,T] \times [a,b]$ is asymptotically between $cT$ and $C T^2$, with
positive constants $c$ and $C$. We also supply with conditions (in terms of
the spectral measure) under which the variance asymptotically grows linearly with $T$,
as a quadratic function of $T$, or has intermediate growth.
\end{abstract}

\section{Introduction}\label{sec: intro}
Consider complex Gaussian Analytic Functions (GAFs) whose distribution is invaruiant under real shifts. Paley and Wiener dedicate the last chapter of their celebrated treatise~\cite[Ch. X]{PW} to study the asymptotic number of zeroes of such functions in a long horizontal rectangle $[-T,T] \times [a,b]$ as $T$ tends to infinity. They show that under certain spectral conditions this quantity satisfies a law of large numbers. This result was extended in~\cite{Naomi}, where unnecessary conditions were removed, and a formula for the expected number of zeroes was developed.

This paper is devoted to the study of the fluctuations in the number of zeroes in the same setting. We study the asymptotic growth of the variance $V(T)$ of the number of zeroes in
$[-T,T]\times [a,b]$.
In particular we identify which processes demonstrate quadratic growth of $V(T)$ (Theorem~\ref{thm: var quadratic}), and supply sufficient and necessary conditions for $V(T)$ to be asymptotically linear (Theorem~\ref{thm: var linear} and~\ref{thm: var super} respectively). In Theorem~\ref{thm: var linear} we also prove that $V(T)$ is always at least linear in $T$. 
Obtaining lower bounds on the variance is the most difficult part of this paper (as was the case in related models, see discussion below).

\subsection{Discussion}\label{sec: disc}

Zeroes of random Gaussian functions in increasing domains have been studied in various settings.
Generally analyizing fluctuations is much harder than calculating expectations and
even than obtaining central limit theorems. 
 For \emph{real Gaussian stationary processes}, a formula for the expectation was obtained in the early works of Kac~\cite{Kac} and Rice~\cite{Rice} from the 1940's. Cramer and Leadbetter~\cite{CrLead} later obtained an asymptotic formula for the variance, but the rate of growth could not be inferred from it. 
In 1976 Cuzick~\cite{Cuz} was able to show that, under some technical assumptions, if the growth of the variance is linear then the number of zeroes in $[-T,T]$ satisfies a Central Limit Theorem (CLT). It was only fifteen years later that Slud~\cite{Slud} obtained
accessible conditions (in terms of the covariance kernel) for this assumption to hold. 
To do so Slud used primarily a sophisticated method for stochastic integration which he developed jointly with Chambers~\cite{ChamSlud}.
Unfortunately, this method involves many computations specifically tailored to tackle the real case, which do not generalize
to the complex setting easily. Consequently, while some of our results for the complex case are very similar to those of Slud (see Remark~\ref{rmk: Slud} below), our methods are
quite different, and, in a sense, more elementary.

Similar methods to those of Cuzick and Slud were recently used by Granville and Wigman~\cite{GW}, to
study the number of zeroes of a \emph{Gaussian trigonometric polynomial} of large degree $N$ in the
interval $[0, 2\pi]$. They showed that the variance of this number is linear in $N$ and that a CLT holds.

In the complex setting Sodin-Tsirelson~\cite{ST1} and Nazarov-Sodin~\cite{SN} studied the unique \emph{planar GAF} whose zeroes are invariant under all planar isometries.
They showed linear growth of the variance and a CLT for the zeroes in large balls (as the radius approaches infinity), using diagram counting methods.

In contrast to these works, this paper requires but standard tools from harmonic analysis, classical analysis and probability.
In order to acquire our results we derive an asymptotic formula for $V(T)$ (Proposition~\ref{prop: var formula}). This formula consists of a series of non-negative terms involving the spectral measure, and is therefore relatively easy to analyse. 
Ideas from this paper were already used for studying the winding number of Gaussian functions from $\mathbb{R}$ to $\mathbb{C}$, in our work with Buckley~\cite{BF}.

\subsection{Definitions}
A \emph{Gaussian Analytic Function (GAF)} in the strip $D=D_\D=\{z: \,
|\text{Im}z|<\Delta\}$ is a random variable taking values in the space of
analytic functions on $D$, so that for every $n\in\N$ and every $z_1,\dots,
z_n\in D$ the vector $(f(z_1),\dots,f(z_n))$ has a mean zero complex Gaussian
distribution.

A GAF in $D$ is called {\em stationary}, if it is distribution-invariant with
respect to all horizontal shifts, i.e., for any $t\in\mathbb R$, any
$n\in\mathbb N$, and any $z_1,\dots,z_n\in D$, the random $n$ -tuples
\[
\bigl( f(z_1),\dots,f(z_n) \bigr) \qquad {\rm and} \qquad \bigl(
f(z_1+t),\dots,f(z_n+t) \bigr)
\]
have the same distribution.

For a stationary GAF, the covariance kernel
$$K(z,w)=\E \{ f(z)\overline{f(w)} \} $$
may be written as
\begin{equation*}
K(z,w)=r(z-\overline w).
\end{equation*}
For $t\in \R$, the function $r(t)$ is positive-definite and continuous, and
so it is the Fourier transform of some positive measure $\rho$ on the real
line:
$$r(t) = \FF[\rho](t)=\int_\R e^{-2\pi i t \lambda}d\rho(\lambda).$$
Moreover, since $r(t)$ has an analytic continuation to the strip $D_{2\D}$,
$\rho$ must have a finite exponential moment:
\begin{equation}\label{eq: cond L1}
\text{for each }|\Delta_1|<\Delta,\:\: \int_{-\infty}^{\infty}
e^{2\pi \cdot 2\Delta_1 |\lambda|} d\rho(\lambda) < \infty\,.
\end{equation}
The measure $\rho$ is called the \emph{spectral measure} of $f$. A stationary
GAF is \emph{degenerate} if its spectral measure consists of exactly one
atom.

For a holomorphic function $f$ in a domain $D$, we denote by $Z_f$ the
zero-set of $f$ (counted with multiplicities), and by $n_f$ the
\emph{zero-counting measure}, i.e.,
\[
\forall \p \in C_0(D): \qquad \int_D \p(z) dn_f(z) = \sum_{z\in
Z_f} \p(z),
\]
where $C_0(D)$ is the set of compactly supported continuous functions on $D$.
We use the abbreviation $\displaystyle n_f(B)=\int_B dn_f(z)$ for the number
of zeroes in a Borel subset $B\subset D$.

\subsection{Results}
First, we present a previous result which will serve as our starting point.
This result can be viewed as a ``law of large numbers'' for the zeroes of
stationary functions.

\begin{thma}\label{thm: LLN}\cite[Theorem 1]{Naomi}
Let $f$ be a stationary non-degenerate GAF in the strip $D_\Delta$, where
$0<\Delta\leq\infty$. Let $\nu_{f,T}$ be the non-negative locally-finite
random measure on $(-\Delta,\Delta)$ defined by
\begin{equation*}
    \displaystyle \nu_{f,T}(Y)=\frac{1}{2T}\,n_f([-T,T)\times Y), \:\:Y\subset (-\Delta,\Delta) \text{ measurable.}
\end{equation*}

Then:

\smallskip\par\noindent{\rm (i)}
Almost surely, the measures $\nu_{f,T}$ converge weakly and on every interval
to a measure $\nu_f$ when $T\rightarrow \infty$.

\smallskip\par\noindent{\rm (ii)}
The measure $\nu_f$ is not random (i.e. $\operatorname{var} \nu_f =0 $) if
and only if the spectral measure $\rho_f$ has no atoms.

\smallskip\par\noindent{\rm (iii)}
If the measure $\nu_f$ is not random, then $\nu_f(Y)=\E n_f([0,1]\times Y)$ and it has density:\\
\begin{equation*}
L(y) =\frac{d}{dy} \left(\frac{\int_{-\infty}^\infty \lambda e^{4\pi
y\lambda}d\rho(\lambda)} {\int_{-\infty}^\infty e^{4\pi
y\lambda}d\rho(\lambda)}\right)= \frac 1 {4\pi} \frac{d^2}{dy^2}
\log \left(r(2iy)\right).
\end{equation*}
\end{thma}

In the above and in what follows, the term ``density'' means the Radon-Nikodym
derivative w.r.t. the Lebegue measure on $\R$.

A natural question is, how big are the fluctuations of the number of zeroes
in a long rectangle? More rigorously, define
$$R^{a,b}_T=[-T,T)\times[a,b],\:\: V^{a,b}_f(T)=\var\left[n_f(R^{a,b}_T)\right],$$
where for a random variable $X$ the variance is defined by
$$\var(X) = \E \left(X-\E X\right)^2 .$$
 We are interested in the asymptotic behavior of $\V$ as $T$ approaches
infinity. The next theorems show that $\V$ is asymptotically bounded
between $cT$ and $C T^2$ for some $c,C>0$, and give conditions under which
each of the bounds is achieved. We begin by stating the upper bound result, a
relatively easy consequence of Theorem~\ref{thm: LLN}.

\begin{thm}\label{thm: var quadratic}
Let $f$ be a non-degenerate stationary GAF in a strip $D_\D$. Then for all
$-\D<a<b<\D$ the limit
  $$\ds L_2=L_2(a,b):= \lim_{T\to\infty}\frac{\V}{T^2} \in [0,\infty)$$
  exists. This limit is positive if and only if the spectral measure of $f$
  has a non-zero discrete component.
\end{thm}


The lower bound result, which is our main result, is stated in the following
theorem.
\begin{thm}\label{thm: var linear}
Let $f$ be a non-degenerate stationary GAF in a strip $D_\D$. Then for all
$-\D<a<b<\D$ the limit
  $$\ds L_1=L_1(a,b):=\lim_{T\to\infty}\frac{\V}{T} \in (0,\infty]$$
  exists. Moreover, the limit $L_1(a,b)$ is finite if $\rho$ is
    absolutely continuous with density $d\rho(\lm)=p(\lm)d\lm$, such that
\begin{align}\label{eq: cond L2}
(1+\lm^2) e^{2\pi \cdot 2y\lambda}p(\lambda) \in L^2 (\R), \text{  for } y\in\{a,b\}.
\end{align}
\end{thm}

Several remarks are due before continuing.
\begin{rmk}\label{rmk: cond L2}
{\rm Another form of condition~\eqref{eq: cond L2} is the following: For
$y\in \{2a,2b\}$,
\begin{equation*}
\int_\R |r(x+iy)|^2 dx,  \:\: \int_\R |r^{\prime \prime}(x+iy)|^2 dx < \infty.
\end{equation*}
This implies also that $\int_\R |r^\prime(x+iy)|^2 dx <\infty$. Moreover, since the set
$\{c:\ e^{2\pi\cdot c \lm }p(\lm) \in L^2 (\R) \}$ is convex, it implies the
same condition for all $y\in [2a,2b]$. }
\end{rmk}

\begin{rmk}\label{rmk: Slud}
{\rm
 It is interesting to note that condition~\eqref{eq: cond L2} is precisely the condition that Slud gave in~\cite{Slud}
for linear variance in the case of real (non-analytic) stationary Gaussian processes (with $a=b=0$).
Nonetheless, no direct implication between the results is known and the methods to obtain them are quite remote.

As for the first part of Theorem~\ref{thm: var linear}, we would expect an analogue to hold for real startionary Gaussian
functions; that is, that under mild conditions, the variance of the number
of zeroes in $[-T,T]$ is always at least linear in $T$. To the best of our knowledge, this is yet unknown.
}
\end{rmk}

\begin{rmk}\label{rmk: final L1}
{\rm In case condition~\eqref{eq: cond L2} holds, we shall give an expression for the limit $L_1$
as a convergent series of the form:
\begin{equation*}
\ds \lim_{T\to\infty}\frac{V_f^{a,b}(T)}{2T} = \sum_{k\ge 1} \int_\R
\left(p^{*k}(\lm)\right)^2 w^{a,b}_k(\lm)d\lm.
\end{equation*}
Here $p^{*k}$ denotes the $k$-fold convolution of $p$, and $w^{a,b}_k(\lm)$
is a positive function which can be computed explicitly in terms of $a$, $b$,
$k$ and $r(2ia)$, $r(2ib)$ (and no other reliance on $p$). }
\end{rmk}

The next theorem deals with conditions under which $L_1(a,b)$ is infinite, i.e. the variance is super-linear.
\begin{thm}\label{thm: var super}
Let $f$ be a non-degenerate stationary GAF in a strip $D_\D$.
\begin{enumerate}[label={\rm (\roman{*})}]
\item\label{item: all a,b}
Suppose $J\subset (-\D,\D)$ is a closed interval such that for every $y\in J$, the function
$\lambda \mapsto (1+\lambda^2) e^{2\pi \cdot 2y\lambda} p(\lambda)$
does not belong to $L^2(\R)$.
Then for every $\alpha\in J$ the set $\{\beta\in J:\ L_1(\alpha,\beta)<\infty\}$ is
finite.

\item\label{item: particular a,b}  The limit $L_1(a,b)$ is infinite for
    particular $a, b$ if either $\rho$ does not have density, or, if it has density $p$
       and for any two points $\lm_1, \lm_2\in \R$
   there exists intervals $I_1,I_2$ such that $I_j$ contains $\lm_j$
   ($j=1,2$) and
   \begin{align}\label{eq: cond inf}
(1+\lm) e^{2\pi \cdot 2y\lambda}p(\lambda) \not\in L^2 (\R\setminus(I_1\cup I_2)),
\end{align}
for at least one of the values $y=a$ or $y=b$.
   \end{enumerate}
\end{thm}

\begin{rmk}\label{rmk: cond not L2}
{\rm There is a gap between the conditions
given for linear variance (in Theorem~\ref{thm: var linear}) and those for
super-linear variance (in Theorem~\ref{thm: var super}). For instance, the theorems do not decide about all the suitable
pairs $(a,b)$ in case the spectral measure has density $\frac 1 {\sqrt
{|\lm|}} \ind_{[-1,1]}(\lm)$. On the other hand, we are ensured to have
super-linear variance in case $\rho$ has a singular part.
If $\rho$ has density $p\in L^1(\R)$ which
is bounded on any compact set, then $(1+\lm^2)p(\lm)\in L^2(\R)$ implies asymptotically
linear variance, and $(1+\lm)p(\lm)\not\in L^2(\R)$ implies asymptotically super-linear variance.
}
\end{rmk}

\begin{rmk}\label{rmk: arg}
{\rm Minor changes to the developments in this paper may be made in order to
prove analogous results regarding the increment of the argument of a
stationary GAF $f$ along a horizontal line. Namely, let $V^{a,a}(T)$ denote
the variance of the increment of the argument of $f$ along the line
$[0,T]\times \{a\}$ (for some $-\D<a<\D$). Then:
\begin{itemize}
\item  the limit $L_2(a)= \lim_{T\to\infty} \frac{V^{a,a}(T)}{T^2}$
    exists, belongs to $[0,\infty)$, and is positive if and only if the
    spectral measure contains an atom.
\item  the limit $L_1(a)=\lim_{T\to\infty} \frac{V^{a,a}(T)}{T}$ exists,
    belongs to $(0,\infty]$, and is finite if $\rho$ has density $p(\lm)$
    such that $(1+\lm^2)e^{2\pi\cdot 2 a\lm }p(\lm)\in L^2(\R)$.
    Moreover, $L_1(a)$ is infinite if for any $\lm_0\in\R$ there is an
    interval $I$ containing $\lm_0$ such that the measure
    $(1+\lm)e^{2\pi \cdot 2a \lm }d\rho(\lm)$ restricted to
    $\R\setminus I$ is not in $L^2(\R)$.
\end{itemize}
In our recent work with Buckley \cite{BF} we extend these statements to hold for a differentiable
(not necessarily analytic) Gaussian process from $\R$ to $\C$.
Also notice that the first item is essentially proved in this paper (Claim~\ref{clm:
error} below). 
}\end{rmk}

The rest of the paper is organized as follows: Theorem~\ref{thm: var
quadratic} concerning quadratic growth of variance is proved in
Section~\ref{sec: square}, and is mainly a consequence of Theorem~\ref{thm:
LLN}. In Section~\ref{sec: form of var} we develop  an asymptotic formula for $\V /T$ (Proposition~\ref{prop: var
formula} below), which will be used to prove Theorems~\ref{thm: var linear} and~\ref{thm: var super}
in Sections~\ref{sec: linear} and~\ref{sec: super} respectively.
Appendices~\ref{sec: off}, \ref{sec: tec} and \ref{sec: CLT} include proofs of some technical lemmas.

\subsection{Acknowledgements}
I thank Mikhail Sodin for his advice and encouragement during all stages of
this work. I am grateful to Jeremiah Buckley, who contributed most of the arguments 
in Appendix~\ref{sec: tec}. I also thank Boris Tsirelson for his interest and suggestions
which stimulated the research, and Boris Hanin for a useful discussion. I am
grateful to Alon Nishry and Igor Wigman for reading the original draft
carefully and pointing out some errors.

\section{Theorem \ref{thm: var quadratic}: Quadratic Variance}\label{sec: square}

Recall the notation $R_T=R^{a,b}_T = [-T,T)\times[a,b].$ From Theorem~\ref{thm: LLN} we know that
$$\ds \lim_{T\to \infty} \frac {n_f(R_T) }{T} = Z,$$ where $Z$ is some random variable and the limit is in the almost sure sense.
Moreover, $\var Z>0$ if and only if the spectral measure of $f$ contains an
atom. Clearly
$$\var \left(\lim_{T\to\infty} \frac {n_f(R_T)}{T} \right) = \var Z$$

Theorem~\ref{thm: var quadratic} would be proved if we could change the limit
with the variance on the left-hand side. By the dominanting convergence theorem, it is
enough to find an integrable majorant for the tails of
$$X_T=\frac {n_f(R_T)}{T}\text{  and  }X_T^2=\frac {n_f(R_T)^2}{T^2}.$$
To this end we use the following proposition, which provides uniform exponential
bounds on tails of $X_T$:
\begin{prop}\label{prop: off}
Let $f$ be a stationary GAF in some horizontal strip, then using the notation
above we have
$$\exists C,c>0: \: \sup_{T\ge 1}\Pro(X_T>s) < Ce^{-cs} =h(s).$$
\end{prop}
Since the proof is quite similar to~\cite[Prop. 5.1]{Naomi}, we defer it to Appendix~\ref{sec: off}.
Proposition~\ref{prop: off} implies that
$$ \sup_{T\ge 1}\Pro(X_T^2>s) < Ce^{-c\sqrt s} =h(\sqrt s).$$
Since both $h(s)$ and $h(\sqrt s)$ are integrable on $\R$, we have the
desired majorants. Exchanging limit and variance then yields:
$$\lim_{T\to\infty}\frac {\var\left(n_f(R_T) \right)}{T^2}=\lim_{T\to\infty}\var \left(\frac {n_f(R_T)}{T} \right) = \var \left(\lim_{T\to\infty}\frac {n_f(R_T)}{T} \right)=\var Z ,$$
and the result is proved.

\section{An Asymptotic Formula for the Variance}\label{sec: form of var}
This section is devoted to the derivation of a formula for the variance
$V^{a,b}_f(T)=\var n_f ([-T,T]\times[a,b])$ where $T$ is large. We prove the
following:

\begin{prop}\label{prop: var formula}
Let $f$ be a stationary GAF in $D_\D$ with spectral measure $\rho$. Suppose
$\rho$ has no discrete component. Then for any $-\D<a<b<\D$, and any
$T\in\R$, the series
$$v^{a,b}(T)=\frac 1{4\pi^2}\sum_{k\ge 1} \frac 1{k^2}\int_{\R}\int_\R T\sinc^2 \left(2\pi T(\lm-\tau) \right)
h^{a,b}_k(\lm+\tau) d\rho^{*k}(\lm)d\rho^{*k}(\tau)$$ converges, and
$$\lim_{T\to\infty} \left(\frac {V^{a,b}(T)}{2T}-v^{a,b}(T) \right) =0. $$
Here $\rho^{*k}$ is the $k$-fold convolution of $\rho$, $\sinc(x)=\frac{\sin
x}x$, and
$$h^{a,b}_k(\lm)=\left(l^a_k(\lm)e^{2\pi
a\lm}-l^b_k(\lm)e^{2\pi b\lm}\right)^2,$$ where
$$l^y_k(\lm)=\pone{y}\left(\frac 1
{r^k(2iy)}\right)+\frac {2\pi}{r^k(2iy)}\lm,  \text{  for }
y\in(-\D,\D), k\in\N.$$
\end{prop}

We begin with some definitions and facts that will be needed along the proof.

\subsection{Preliminaries}\label{sec: formula pre}

\subsubsection{Tools from Harmonic Analysis}\label{sec: Fourier}
In this section we discuss some operations on measures and their relation to the Fourier transform.

Denote by $\M$ the space of all finite measures on $\R$, similarly $\Mplus$
denotes all finite non-negative measures on $\R$.
For two measure $\mu, \nu\in \M$ the \emph{convolution} $\mu *\nu\in \M$ is a
measure defined by:
$$\ds \forall \p\in C_0(\R):\: (\mu*\nu)(\p)=\iint \p(\lm+\tau)d\mu(\lm)d\nu(\tau).$$
When both measures have density, this definition agrees with the standard
convolution of functions. We write $\mu^{*k}$ for the iterated convolution of
$\mu$ with itself $k$ times.

for a measure $\mu\in \Mplus$
having exponential moments up to $2\D$ (i.e., obeying condition~\eqref{eq:
cond L1}), and a number $y\in (-2\D,2\D)$, we define the \emph{exponentially
rescaled measure} $\mu_y\in \Mplus$ by
\[\forall \p\in C_0(\R): \: \mu_y(\p)=\mu(e^{2\pi y\lm } \p(\lm)) = \int_\R e^{2\pi y\lm }\p(\lm)d\mu(\lm) \]

\begin{obs*}
  For any $\mu, \nu\in \M$ and any $|y|<2\Delta$,
  $$ (\mu*\nu)_y = \mu_y*\nu_y.$$
\end{obs*}
\begin{proof}
for any test function $\p \in C_0(\R)$ we have:
\begin{align*}
  \int \p \ d(\mu_y*\nu_y) &= \iint \p(\lm+\tau) \ d\mu_y(\lm)d\nu_y(\tau)\\
  &= \iint \p(\lm+\tau) e^{2\pi y(\lm+\tau)}\ d\mu(\lm)d\nu(\tau) = \int\p \ d(\mu*\nu)_y
  \end{align*}
\end{proof}
\begin{cor*}
If $\rho\in \Mplus$ is such that~\eqref{eq: cond L1}, then for any $|y|<2\D$ and $k\in \N$ we have
$ (\rho_y)^{*k} = (\rho^{*k})_y,$
so there will be no ambiguity in the notation $\rho^{*k}_y$.
\end{cor*}

Next, we define for $\mu\in \M$ the \emph{flipped measure} $\flip{\mu}\in\M$
by:
$$\flip{\mu}(I)=\mu(-I) \text{ for any interval } I\subset\R,$$
and the \emph{cross-correlation} of measures $\mu,\nu\in \M$ by:
$$\mu\star \nu := \mu * \flip{\nu}.$$
An alternative definition, via actions on test-functions, would be:
$$\ds \forall \p\in C_0(\R):\: (\mu\star\nu)(\p)=\iint \p(\lm-\tau)d\mu(\lm)d\nu(\tau).$$
Notice that the cross-correlation operator is bi-linear, but not commutative. In all following expressions, convolution precedes cross-correlation.

We are now ready to prove a lemma relating these notions to the Fourier tranform.

\begin{lem}\label{lem: inv exists}
Suppose $\rho\in\Mplus$ obeys~\eqref{eq: cond L1}, and $r=\FF[\rho]$. Then, for any $|y|<2\D$, $x\in\R$ and $k\in\N$:
$$|r(x+iy)|^{2k}=\FF [ \rho^{*k}_y \star  \rho^{*k}_y ](x).$$

This measure acts on a test-function
 $\p\in C_0(\R)$ in the following way:
\begin{align*}
 (\rho^{*k}_y \star  \rho^{*k}_y)(\p)
&= \iint \p(\lm-\tau)e^{2\pi y(\lm+\tau)
}d\rho^{*k}(\lm)d\rho^{*k}(\tau).
\end{align*}
\end{lem}

\begin{proof}
Fix $k\in\N$. Since $r(z)=\FF[\rho](z)$, by a standard property of Fourier transform one has
$r^k(z)=\FF[\rho^{*k}](z)$.
Writing $z=x+iy$, this reads
\begin{equation*}
r^k(x+iy)=\int_\R e^{-2\pi ix \lm }e^{2\pi y \lm} d\rho^{*k}(\lm).
\end{equation*}
This implies:
\begin{itemize}
\item  $r^{k}(x+iy)= \FF[\rho^{*k}_y](x)$
\item
    $\overline{r^k(x+iy)}=\FF[\rho^{*k}_y](-x)=\FF[\flip{\rho^{*k}_y}](x)$,
\end{itemize}
which leads to
\[
|r(x+iy)|^{2k} =  \FF[\rho^{*k}_y * \flip{\rho^{*k}_y}](x) = \FF[\rho^{*k}_y \star \rho^{*k}_y](x).
\]
\end{proof}

Also useful to us will be the following special case of Parseval's identity for measures (see Katznelson ~\cite[VI.2.2]{Kat}):
\begin{lem}\label{lem: Parseval}
For any finite measure $\g$ on $\R$,
\begin{equation*}
\ds \int_{-2T}^{2T}\left(1-\frac{|x|}{2T}\right) \FF[\g](x) dx=
\int_\R 2T \sinc^2 (2\pi T\xi) d\g(\xi).
\end{equation*}
where $\sinc(\xi)=\frac{\sin \xi}{\xi}$ and $\FF[\g]$ is the Fourier
transform of $\g$.
\end{lem}

\subsubsection{Properties of a ``normalized'' covariance function}
Here we summarize some properties of a normalized version of the covariance function, namely
$\frac{|r(x+i a+i b)|^2}{r(2ia)\,r(2ib)},$
which will be used later in our proofs.
 In the following, when we do not specify the variables
we mean the statements holds on all the domain of definition. We use the
subscript notation for partial derivatives (such as $q_a$ for $\pone{a} q$).


\begin{lem}\label{lem: q}
The function
\begin{equation*}
 q(x,a,b) := \frac{|r(x+i a+i b)|^2}{r(2ia)\,r(2ib)},
 \end{equation*}
 is well-defined, infinitely differentiable on $\R\times(-\D,\D)^2$, and
satisfies the following properties:
\begin{enumerate} [label=\textbf{\arabic{*}. }]
  \item $q(x,y_1,y_2)\in [0,1]$. \\
$q(x,y_1,y_2)=1\text{ if and only if } (x=0 \text{ and } y_1=y_2).$

\item $\sup_{x\in\R } q(x,y_1,y_2)<1$
for any $y_1\ne y_2$ in $(-\D, \D)$.
\item For fixed $y_1$ and $y_2$ let $g_{y_1,y_2}(x)$ be one of the
    functions $q$, $q_a$, $q_b$, $q_{ab}$ evaluated on the line
    $\{(x,y_1,y_2): \, x\in \R\}$. Then $g_{y_1,y_2}\in L^\infty(\R)$.
If condition~\eqref{eq: cond L2} holds, then for any $y_1,y_2\in [a,b]$
    we have also $g_{y_1,y_2}\in L^1(\R) \cap C_0(\R)$ (i.e., is integrable and tends to zero as $x\to\pm \infty$).
\item $q_a(0,t,t)=0$, for any $t\in(-\D,\D)$.
\end{enumerate}
\end{lem}

\begin{proof}
Since $r(2iy)>0$ for all $y\in\R$, the function $q$ is indeed well-defined;
differentiability follows from that of $r(z)$.

For item 1, notice that
  $$q(x,a,b)
   = \frac {\left(\int e^{2\pi(a+b) \lm}e^{-2\pi i x\lm}d\rho(\lm) \right)^2}{\int e^{2\pi\cdot 2a \lm}d\rho(\lm)\, \int e^{2\pi\cdot 2a \lm}d\rho(\lm)}$$
  and so, by Cauchy-Schwarz, is in $[0,1]$. Equality $q(x,a,b)=1$ holds only if the
  function $\lm\mapsto e^{2\pi\cdot a \lm} e^{-2\pi i x\lm}$ is a constant times the function
  $\lm\mapsto e^{2\pi\cdot b \lm}$, $\rho$-a.e., but, if $\rho$ is non-atomic, this is impossible unless $x=0$ and
  $a=b$.

Further, we notice that
$$|r(x+ia+ib)|  = \left|\int e^{2\pi(a+b) \lm}e^{-2\pi i x\lm}d\rho(\lm)  \right| \le \int e^{2\pi(a+b)\lm} d\rho(\lm) = r(ia+ib),$$
so that $q(x,a,b) \le q(0,a,b)<1$ (the right-most inequality is by item 1). Taking the supremum yields item 2.

For item 3, notice any one of the functions $q, q_a, q_b, q_{ab}$ is the sum
of summands of the form
\begin{equation}\label{eq: term}
C(a,b)\  r^{(j)}(x+i a+i b)\ r^{(m)}(-x+i a+i b),
\end{equation}
where $0\le j, m\le 2$ are integers. It is enough therefore to explain why
$r^{(j)}(x+ia+ib)$ is bounded and approaches zero as $x\to\pm \infty$, for
any integer $0\le j\le 2$. Recall that
$$r^{(j)}(x+iy)= c_j \FF[\lm^j e^{2\pi y\lm} \,d\rho(\lm)](x), $$
where $c_j$ is some constant. As a function of $x$, this is a Fourier
transform of a non-atomic measure, therefore has the desired properties.

If condition~\eqref{eq: cond L2} holds, then $d\rho(\lm)=p(\lm)d\lm$, and the
function $\lm\mapsto\lm^j e^{2\pi (y_1+y_2)\lm}p(\lm)$ is in $L^2(\R)$. Then,
its Fourier transform $r^{(j)}(x+iy_1+iy_2)$ is also in $L^2(\R)$, and each
summand of the form~\eqref{eq: term} is in $L^1(\R)$, as anticipated.

For item 4, notice that for all $x\in\R$ and all $a,b\in(-\D,\D)$ we have the
symmetry $q(x,a,b)=q(x,b,a)$, and therefore for all $t\in\R$:
$q_a(x,t,t)=q_b(x,t,t)$. On the other hand, for all $t\in (-\D,\D)$ it holds
that $q(0,t,t)=1$, so taking derivative by $t$ we get $q_a(0,t,t)\cdot 1
+q_b(0,t,t)\cdot 1 =0. $ This proves the result.
\end{proof}

We are now ready to begin the proof of Proposition~\ref{prop: var formula}.

\subsection{Integrals on significant edges.}\label{sec: significant edges}
 The boundary of
the rectangle $R_T=[-T,T]\times[a,b]$ is composed of four segments $\partial
R_T=\bigcup_{1\le i\le 4}I_j$ with induced orientation from the
counter-clockwise orientation of $\partial R_T$, where
$I_1=[-T,T]\times\{a\}$ and $I_3=[T,-T]\times\{b\}$. By the argument
principle,
$$n_f(R_T) = \sum_{1\le i\le 4} \frac 1 {2\pi}\triangle^T_i \arg f,$$
where $\triangle^T_i \arg f$ is the increment of the argument of $f$ along
the segment $I_i$ (a.s. $f$ has no zeroes on the boundary of the rectangle
$R_T$ \footnote{ To see this, first notice that the distribution of
$n_f(I_j)$ for $j=2,4$ (the number of zeroes in a ``short'' vertical
segments) does not depend on $T$. If it were not a.s. zero, then $\E
n_f(I_2)>0$. Now for any finite set of points $\{t_j\}_{j=1}^N\subset [0,1]$,
we have $\E n_f([0,1]\times [a,b])\ge \sum_{j=1}^N \E n_f(\{t_j\}\times
[a,b])= N\E n_f(I_2)$, yielding $\E n_f([0,1]\times [a,b])=\infty$ - which is
false. For $j=1,3$, recall that since there are no atoms in the spectral
measure, $f$ is ergodic with respect to horizontal shifts (this is
Fomin-Grenander-Maruyama Theorem, see explanation and references
within~\cite{Naomi}). This implies that each horizontal line (such as $L_a=\R
\times \{a\}$) either a.s. contains a zero or a.s. contains no zeroes. If the
former holds, then also $\E n_f([0,1]\times \{a\})>0$, and the measure
$\nu_f$ from Theorem~\ref{thm: LLN} has an atom at $a$ - contradiction to
part (iii) of that Theorem.}).

Then, by the argument principle,
\begin{equation}\label{eq: four segs}
V^{a,b}_f(T)=\var \left[n_f ( R_T) \right]=\frac 1 {4\pi^2}
\sum_{1\le i,j\le 4} \cov \left(\triangle^T_i \arg f,\ \triangle^T_j
\arg f \right), \end{equation} where
$$\cov(X,Y)=\E[XY]-\E X\cdot \E Y.$$

 Our first claim is that asymptotically when $T$ is large, the
terms involving the (short) vertical segments are negligible in this sum.

\begin{clm}\label{clm: big segments}
As $T\to\infty$, one has:
\begin{align*}
\V &= \frac 1 {4\pi^2} \sum_{i,j\in\{1,3\}} \cov \left(\triangle^T_i
\arg f,\ \triangle^T_j \arg f \right)
+O\left(1+\sqrt{\var (\triangle^T_1 \arg f)} + \sqrt{\var (\triangle^T_3
\arg f)}\right).
\end{align*}
\end{clm}
\begin{proof}
We demonstrate how to bound one of the terms in~\eqref{eq: four segs}
involving a ``short'' vertical segment (corresponding, say, to $i=2$).
By stationarity,
$\var (\triangle^T_2 \arg f)= \var (\triangle^0_2 \arg f)=:c^2$.
Applying the Cauchy-Schwarz inequality, we have:
\begin{align*}
\cov \left( \triangle^T_1 \arg f,\ \triangle^T_2 \arg f \right)
&\le \sqrt{\var (\triangle^T_1 \arg f)}\sqrt{\var (\triangle^T_2 \arg f)}\\
&= c\cdot \sqrt{\var (\triangle^T_1 \arg f)}.
\end{align*}
\end{proof}

Let us now give an alternative formulation of Claim~\ref{clm: big segments}. Using Cauchy-Riemann equations, we have:
\begin{align*}
\triangle^T_1 \arg f &= \intt \left(\frac{\partial}{\partial x}\arg f(x+ia)\right) dx=-\intt \frac{\partial}{\partial a}\log|f(x+ia)|\: dx=:-X^a(T)\\
\triangle^T_3 \arg f &= -\intt \left(\frac{\partial}{\partial x}\arg
f(x+ib)\right) dx=\intt \frac{\partial}{\partial b}\log|f(x+ib)|\:
dx=X^b(T)
\end{align*}

Denoting $C^{a,b}(T)=\cov (X^a(T),X^b(T))$ we may rewrite Claim~\ref{clm: big
segments} as
\begin{equation*}
  \V =\frac 1{4\pi^2}\left(C^{a,a}(T)-2C^{a,b}(T)+C^{b,b}(T)\right) + O\left(1+\sqrt{C^{a,a}(T)}+\sqrt{C^{b,b}(T)}\right),
\end{equation*}
or alternatively:

\begin{customclm}{1a}\label{clm: var in terms of cov}
As $T\to\infty$, we have:
\begin{equation*}
\frac {V^{a,b}(T)}{2T}=\frac
{C^{a,a}(T)-2C^{a,b}(T)+C^{b,b}(T)}{4\pi^2\cdot 2T}+O\left(
\frac{ 1+ \sqrt{C^{a,a}(T)} +\sqrt{C^{b,b}(T)} }{T}\right).
\end{equation*}
where
\begin{equation}\label{eq: C}
\begin{split}
C^{a,b}(T)=&\E \left\{ \intt dt \,\intt ds\, \left(\pone{a}\log |f(t+ia)|\ \pone{b}\log |f(s+ib)| \right) \right\}\\
&\quad -\E\left\{\intt \pone{a} \log|f(t+ia)| \: dt\right\}\E\left\{\intt \pone{b} \log|f(s+ib)| \: ds\right\}\\
\end{split}
\end{equation}
\end{customclm}

\subsection{Changing order of operations}\label{sec: exchanges}
Our goal now is to prove the following:
\begin{clm}\label{clm: Fubini}
\begin{equation}\label{eq: before series}
\ds C^{a,b}(T)= \intt  \intt  \
\ptwo{a}{b} \cov\left(\log|f(t+ia)|,\ \log|f(s+ib)| \right)dt \ ds.
\end{equation}
\end{clm}

The meaning of this formula for $C^{a,a}(T)$ is as follows:
on the RHS, first take the mixed partial derivative (as if $a\ne b$), then substitute $b=a$ and integrate by $t$ and $s$.

The proof is an application of the following two lemmas.
In both, we assume $f$ is a stationary GAF in $D_\D$, and $a,b\in (-\D,\D)$.
\begin{lem}\label{lem: I,II}
For any $T>0$ the following integrals are finite:
\begin{enumerate}[label={\rm (A-\Roman{*})}]
\item\label{item: one dim}
\quad $\displaystyle \intt \E \left|
    \frac{f'(t+ia)}{f(t+ia)}\right| dt <\infty.$
\vspace{5pt}
\item\label{item: two dim}
\quad $\displaystyle\intt \intt \E \left|
    \frac{f'(t+ia)}{f(t+ia)} \frac{f'(s+ib)}{f(s+ib)} \right| dt \ ds
    <\infty.$
\end{enumerate}
\end{lem}

\begin{lem}\label{lem: deriv}
For almost all $t,s\in [-T,T]^2$,
\begin{enumerate}[label={\rm (B-\Roman{*})}]
\item   \label{eq: j1}
\quad $\displaystyle \E \left[ \pone{a} \log |f(t+ia)| \right]= \pone{a} \E \bigg[ \log |f(t+ia)| \bigg]$
\vspace{5pt}
\item\label{eq: j2}
\quad $\displaystyle \E \left[ \ptwo{a}{b}  \log |f(t+ia)| \log |f(s+ib)| \right]  = \ptwo{a}{b} \E  \bigg[ \log |f(t+ia)| \log |f(s+ib)| \bigg]. $
\end{enumerate}
\end{lem}

Though simple to state, Lemmas~\ref{lem: I,II} and \ref{lem: deriv} both require long and non-trivial proofs, which do not follow any standard method we are aware of. We therefore defer their proofs to Appendix~\ref{sec: tec}.

\begin{proof}[Proof of Claim~\ref{clm: Fubini}]
Recall the definition of $C^{a,b}(T)$ in~\eqref{eq: C}, and notice that
\begin{equation}\label{eq: d log}
\left|\pone{a} \log|f(x+ia)|\right| \le \left|\frac {f'(x+ia)}{f(x+ia)}\right|.
\end{equation}
{\it Step 1:} Exchange integrals and expcectation: By~\eqref{eq: d log} and Lemma \ref{lem: I,II}, we may apply Fubini's theorem to get:
\begin{equation*}
\begin{split}
C^{a,b}(T)=& \intt \intt \ \big[ \,  \E \ptwo{a}{b}  \, \{ \log|f(t+ia)|\,
\log|f(s+ib)| \}
\\
&\hspace{3.7cm}-\E \pone{a} \log|f(t+ia)|\, \E \pone{b} \log|f(s+ib)|\ \big] dt \ ds .
\end{split}
\end{equation*}
{\it Step 2: } Exchage the order of expectation and derivative inside the integral by $t$ and $s$.
This is justified directly by Lemma~\ref{lem: deriv}.
We arrive at the desired form.
\end{proof}

\subsection{The error term}\label{sec: error}
Next, we show that the error term in Claim~\ref{clm: var in terms of cov} approaches zero as $T$ tends to infinity.

\begin{clm}\label{clm: error}
If $\rho$ contains no atoms, then for any $a\in (-\D,\D)$:
$$\lim_{T\to\infty} \frac{C^{a,a}(T)}{T^2}=0. $$
\end{clm}

\begin{proof}
Since $\rho$ has no atoms, $f$ is an ergodic process (this is the classical Fomin-Grenander-Maruyama theorem, see~\cite[Theorem 4]{Naomi} and references therein).
Thus, by the ergodic theorem,
\begin{equation}\label{eq: erg}
\lim_{T\to\infty}\frac 1 T X^a(T) = \E X^a(1)
\end{equation}
converges almost surely and in $L^1$ to a constant.
Recall $X^a(T)$ has finite second moment (this is precisely relation~\ref{item: two dim}).
Therefore, the convergence in~\eqref{eq: erg} is also in the $L^2$ sense (this is an easy adaptation of the proof for $L^1$ convergence, see \cite[Exercise 7.2.1]{Dur}).
We conclude that:
$$\lim_{T\to\infty}\frac 1 {T^2} \var (X^a(T)) = \lim_{T\to\infty}\frac 1 {T^2} C^{a,a}(T) =0.$$
\end{proof}

Claims~\ref{clm: var in terms of cov} , \ref{clm: Fubini} and \ref{clm: error} give together:
\begin{cor*}
\begin{equation}\label{eq: with error}
\frac {V^{a,b}(T)}{2T}=\frac{C^{a,a}(T)-2C^{a,b}(T)+C^{b,b}(T)}{4\pi^2\cdot 2T}+o(1), \quad T\to\infty,
\end{equation}
where $C^{a,b}(T)$ is given by~\eqref{eq: before series}.
\end{cor*}

\subsection{A formula in terms of the covariance function}
Our goal now is to replace~\eqref{eq: before series} by a simpler formula, using the covariance function.
This is done in the next claim.

\begin{clm}\label{clm: for Fourier}
\begin{align}\label{eq: for Fourier}
  \frac{C^{a,b}(T)}{2T}
  & = \frac 1 2  \int_{-2T}^{2T}\left(1-\frac {|x|}{2T}\right)
  \sum_{k\ge 1} \ptwo{a}{b}  \frac {q(x,a,b)^k}{k^2}\ dx,
\end{align}
where
 \begin{equation}\label{eq: q}
 q(x,a,b) := \frac{|r(x+i a+i b)|^2}{r(2ia)\,r(2ib)}.
 \end{equation}
\end{clm}

For the proof, we will need the following Lemma, which is a direct consequence of a lemma by Nazarov and Sodin~\cite[Lemma
2.2]{SN}  (see also \cite[Lemma 3.5.2]{GAF book}).
\begin{lem} \label{lem: cov of logs}
If $\xi$ and $\eta$ are centered complex
Gaussian random variables, then:
\[
\cov(\log|\xi|,\log|\eta|)=\frac 1 4\sum_{k=1}^\infty
\frac{1}{k^2}\left(\frac
{|\E(\xi\overline{\eta})|^2}{\E|\xi|^2\E|\eta|^2}\right)^{k}.
\]
\end{lem}

\begin{proof}[Proof of Claim~\ref{clm: for Fourier}]
Taking $\xi=f(t+ia)$ and $\eta=f(s+ib)$, we have due to
stationarity:
\begin{equation*}
\frac
{|\E(\xi\overline{\eta})|^2}{\E|\xi|^2\E|\eta|^2} = \frac {|\E (f(t+ia)\overline{f(s+ib)} )|^2}
{\E|f(t+ia)|^2\,\E|f(s+ib)|^2}=\frac {|r(t-s+ia+ib)|^2}
{r(2ia)\,r(2ib)}=q(t-s,a,b).
\end{equation*}
Therefore, by Lemma~\ref{lem: cov of logs} equation~\eqref{eq: before series} becomes:
\begin{align*}
C^{a,b}(T)
 &= \frac 1 4 \intt  \intt  \
 \left\{\ptwo{a}{b}  \sum_{k=1}^\infty \frac 1{k^2}q(t-s,a,b)^k \right\} \, dt \, ds \\
& = \frac 1 2  \int_{-2T}^{2T}\left(1-\frac {|x|}{2T}\right)
   \left\{ \ptwo{a}{b}  \sum_{k\ge 1}  \frac {q(x,a,b)^k}{k^2} \right\}  \ dx.
\end{align*}
In the last equality we used that $\intt \intt Q(t-s) dt\, ds = 2\int_{-2T}^{2T}(2T-|x|) Q(x) dx$
for any $Q\in L^1([-2T,2T])$), which can be proved by a simple change of variables.

All that remains in order to get~\eqref{eq: for Fourier}, is to prove that
\begin{equation}\label{eq: d sum}
\forall x\ne 0, \quad \ptwo{a}{b} \sum_{k\ge 1}  \frac {q(x,a,b)^k}{k^2} =
\sum_{k\ge 1}   \ptwo{a}{b} \frac {q(x,a,b)^k}{k^2}
\end{equation}

Fix $x\ne 0$. For shortness, we do not write the variables
$(x,a,b)$, and use again the subscript notation for partial derivatives. We compute:

\begin{equation*}
  S^{a,b}_k(x):=\ptwo{a}{b} \left\{q^k\right\} =
  \begin{cases}
    q_{ab}, & k=1\\
    k(k-1)q^{k-2}q_a q_b, + kq^{k-1}q_{ab} & k>1.
  \end{cases}
\end{equation*}
Therefore,
\begin{equation*}
\left|\frac {S^{a,b}_k(x)}{k^2}\right| \le  q^{k-2} |q_a q_b|
+ \frac 1 k q^{k-1} |q_{ab}|.
\end{equation*}

By part 1 of Lemma~\ref{lem: q}, $q(x,a,b) < 1$ (notice this holds also if $a=b$).
We deduce that $\sum \left|\frac{S_k^{a,b}}{k^2} \right| < \infty$.
i.e., the RHS of~\eqref{eq: d sum} converges in absolute value.
By standard arguments, this is enough to prove equality~\eqref{eq: d sum}.
Claim~\ref{clm: for Fourier} follows.
\end{proof}

\subsection{A formula in terms of the spectral measure}\label{sec: deriv}
In this section, we finally prove Proposition~\ref{prop: var formula}, by carefully moving to the spectral representation of the formula we had at hand.

\begin{proof}[Proof of Proposition~\ref{prop: var formula}]
Using Lemma~\ref{lem: inv exists} and the definition of $q$ in~\eqref{eq: q}, we get
\[ q(x, a,b)^k= \frac{\FF[\rho_{a+b}^{*k}\star \rho_{a+b}^{*k}](x)} { r^k(2ia)\, r^k(2ib)}
\]

Define
\begin{equation}\label{eq: Sk1}
S^{a,b}_k(x):=\ptwo{a}{b} \left\{q(x, a,b)^k\right\}
=\ptwo{a}{b} \left\{   \frac{\FF[\rho_{a+b}^{*k}\star \rho_{a+b}^{*k}](x)}{ r^k(2ia)\, r^k(2ib)} \right\}.
\end{equation}

\begin{obs*}
\[
S_k^{a,b}(x)= \FF\left[l^a_k(\lm) \rho_{a+b}^{*k}\star l^b_k(\lm)\rho_{a+b}^{*k}  \right](x).
\]
where $l^a_k(\lm), l^b_k(\lm)$ are linear functions in $\lm$, given by
$$l^a_k(\lm)=\pone{a}\left(\frac 1
{r^k(2ia)}\right)+\frac {2\pi}{r^k(2ia)}\lm =
\frac{2}{r^k(2ia)}\left(-ik \frac{r'(2ia)}{r(2ia)}+\pi \lm\right).$$
\end{obs*}

\begin{proof}
Recall that
\begin{align*}
\FF[\rho_{a+b}^{*k}\star \rho_{a+b}^{*k}](x) &= \iint e^{-ix (\lm -\tau)} e^{2\pi i (a+b) (\lm +\tau)}d\rho^{*k}(\lm) d\rho^{*k} (\tau),
\end{align*}
and notice we may differentiate by $a$ and $b$ under the integral, as the result would be continuous and integrable w.r.t. $\rho^{*k}$. From here,
the proof is a straightforward computation.
\end{proof}

Using this observation, we rewrite equation~\eqref{eq: for Fourier} as follows:
\begin{equation}\label{eq: for end}
\frac {C^{a,b}(T)}{2T}=
 \int_{-2T}^{2T} \left(1-\frac {|x|}{2T}\right)
\sum_{k\ge 1} \frac 1 {2k^2}
\ptwo{a}{b} \left\{ \frac { \FF[\rho_{a+b}^{*k}\star \rho_{a+b}^{*k}](x)}{r^k(2ia) r^k(2ib)}  \right\} \ dx.
\end{equation}

Futher, using Lemma~\ref{lem: Parseval} (Parseval's identity), we have for fixed $k$,
\begin{align}\label{eq: int psi}
 \int_{-2T}^{2T}& \left(1-\frac{|x|}{2T}\right) S_k^{a,b}(x)\ dx  = \int_{-2T}^{2T} \left(1-\frac{|x|}{2T}\right)
 \FF\left[l^a_k(\lm) \rho_{a+b}^{*k}\star l^b_k(\lm)\rho_{a+b}^{*k}  \right](x) \notag \\
&= \int_{\R}\int_\R
2T\sinc^2(2\pi T(\lm-\tau)) \ l_k^a(\lm) l_k^b(\tau) e^{2\pi (a+b)(\lm+\tau)} d\rho^{*k}(\lm)d\rho^{*k}(\tau).
\end{align}

Plugging~\eqref{eq: for end} into~\eqref{eq: with error}, we get:
\begin{align}\label{eq: sum out}
\frac {\V}{2T}=
&\frac 1 {4\pi^2} \int_{-2T}^{2T} \left(1-\frac {|x|}{2T}\right) \sum_{k\ge 1} \frac 1{2k^2}(S_k^{a,a}(x) - 2S_k^{a,b}(x) + S_k^{b,b}(x)) dx +o(1)\notag\\
&=\frac 1 {4\pi^2} \sum_{k\ge 1} \frac 1{2k^2}\int_{-2T}^{2T} \left(1-\frac {|x|}{2T}\right) \ (S_k^{a,a}(x) - 2S_k^{a,b}(x) + S_k^{b,b}(x)) dx +o(1)\\
& =  \frac 1{8\pi^2} \sum_{k\ge 1}\frac 1 {k^2} \int_{\R}\int_\R T\sinc^2(2\pi T(\lm-\tau))
h^{a,b}_k(\lm+\tau) d\rho^{*k}(\lm)d\rho^{*k}(\tau)+o(1), \notag
\end{align}
where
\begin{equation}\label{eq: h}
h^{a,b}_k(\lm)=\left(l^a_k(\lm)e^{2\pi a\lm}-l^b_k(\lm)e^{2\pi
b\lm}\right)^2.
\end{equation}

The exchange of sum and integral in the second equality of~\eqref{eq: sum out} is justified by the monotone convergence theorem, as each term in the series is non-negative.
The last equality follows from~\eqref{eq: int psi}.
Equation~\ref{eq: sum out} establishes Proposition~\ref{prop: var formula}.
\end{proof}

\section{Theorem \ref{thm: var linear}: Linear and Intermediate Variance}\label{sec: linear}
We dedicate Section~\ref{sec: pre} to prove some facts which will be needed along the proof. Afterwards, we prove the
existence of the limit $L_1$ and its positivity in Section~\ref{sec: lower bound}. In Section~\ref{sec: lin finite} we prove $L_1$ is finite under
condition~\eqref{eq: cond L2}.

\subsection{Preparation}\label{sec: pre}

\subsubsection{Tools from Analysis}
First we present two observations about convolutions and integration.

\begin{obs}\label{obs: lim triangle}
If $Q:\R \to [0,\infty)$ is integrable on $\R$, then
$$\lim_{T\to\infty}\int_{-T}^T\left(1-\frac{|x|}{T}\right)Q(x) dx =\int_\R Q.$$
\end{obs}

\begin{proof}
Notice that:
  $$\int_{-\sqrt T}^{\sqrt T}\left(1-\frac 1 {\sqrt T} \right)Q(x) dx
  \le  \int_{-T}^T\left(1-\frac{|x|}{T}\right)Q(x) dx \le \int_{-T}^T Q(x)dx, $$
and both ends of the inequality approach the limit $\int_\R Q$.
\end{proof}

\begin{obs}\label{obs: conv}
For any $\psi_1, \psi_2\in C_0(\R)$ and $\mu\in \Mplus$, $$\int \psi_1 \,
d(\mu*\psi_2) = \int \ (\psi_1*\flip{\psi_2}) \, d\mu. $$
\end{obs}

\begin{proof}
\begin{align*}
\int \psi_1 &\, d(\mu*\psi_2)
= \int \psi_1(x+y) \, d\mu(x) \psi_2(y) dy \\
&= \int \left(\int \psi_1(x+y) \,\flip{\psi_2}(-y) \, dy\right)d\mu(x)
 = \int (\psi_1*\flip{\psi_2})(x) d\mu(x).
\end{align*}
\end{proof}

The following lemma shall play a key-role later on in our proof.
\begin{lem}\label{lem: mu on balls dmu}
Let $\mu\in \Mplus$ ($\mu\not\equiv 0$). Then the following limit exists
(finite or infinite):
$$\lim_{\ep\to 0+} \int_\R \frac 1 {2\ep} \mu\left(\tau-\ep,\tau+\ep\right)
d\mu(\tau)= 
\int_\R |\FF[\mu]|^2(x) dx.$$
\end{lem}

\begin{proof} 
Denote $\p_\ep = \frac 1{2\ep}\ind_{(-\ep,\ep)}$ for $\ep>0$.
 Rewriting the integral and using Parseval's identity, we get:
\begin{align*}
I_\mu(\ep)&:= \frac 1 {2\ep} \int_\R  \mu\left(\tau-\ep,\tau+\ep\right) d\mu(\tau)\\
& = \int_\R (\mu * \p_\ep)(\tau) d\mu(\tau)\\
& = \int_\R (\FF[\mu]\cdot \FF[\p_\ep])(x) \FF[\mu](-x) dx\\
&= \int_\R \sinc(2\pi\ep x) |\FF[\mu]|^2(x) dx
\end{align*}
Since $|\sinc(2\pi \ep x)|\le 1$, we have the upper bound $I_\mu(\ep)\le
\int_\R |\FF[\mu]|^2(x) dx.$

Using the Observation~\ref{obs: conv} and the fact that $\p_\ep * \p_\ep \le 2\p_{2\ep}$
we get:
\begin{align*}
\int_\R (\mu* \p_\ep)\ d(\mu* \p_\ep)  &= \int_\R  \mu *(\p_\ep*\p_\ep) d\mu
 \le 2 \int_\R \mu* \p_{2\ep} \ d\mu =2 I_\mu(2\ep)
\end{align*}

On the other hand,
\begin{align*}
\int_\R (\mu* \p_\ep)\ d(\mu* \p_\ep) &= \int |\FF[\mu*\p_\ep]|^2 = \int  |\FF[\mu]|^2\cdot \sinc^2(2\pi\ep x)\  dx\\
&\ge \int_K |\FF[\mu]|^2\cdot \sinc^2(2\pi\ep x) \ dx
\end{align*}
for any compact set $K\subset\R$. Since the limit $\ds \lim_{\ep\to
0+}\sinc(2\pi\ep x)=1$ is uniform in $x\in K$, the last expression approaches
$\int_K |\FF[\mu]|^2$ as $\ep\to 0+$. Thus, by choosing $K$ and then $\ep>0$
properly, the lower bound may be made arbitrarily close to $\int_\R
|\FF[\mu]|^2$. This concludes the proof.
\end{proof}

\subsubsection{Lower bounds on $h^{a,b}_1$}
The main goal of this subsection is to give lower bounds on $h^{a,b}_1$ (defined in~\eqref{eq: h}).
We begin with a simple claim.
\begin{clm}\label{clm: two zeroes}
The function $h^{a,b}_1$ has exactly two real zeroes.
\end{clm}

\begin{proof}
By the form of $h^{a,b}_1$, $h^{a,b}_1(\lm)=0$ if and only if
$$e^{2\pi (b-a)\lm} = \frac{l^a_1(\lm)}{l^b_1(\lm)} = \frac{ \frac{1}{r(2ia)}\left(\pi \lm -i \frac{r'}{r}(2ia) \right)  }
{\frac{1}{r(2ib)}\left(\pi \lm -i \frac{r'}{r}(2ib)\right)} = C\cdot
\frac{\lm-\psi(a)}{\lm-\psi(b)}, $$ where $C>0$ is a positive
constant and $\psi(y)= \frac 1 {2\pi} \frac{d}{dy}\left[\log r(2iy) \right]$.
Since $y\mapsto\log r(2iy)$ is a convex function, for $a<b$ we have
$\psi(a)<\psi(b)$. Therefore, $\lm\mapsto C\ \frac{\lm-\psi(a)}{\lm-\psi(b)}$
is a strictly decreasing function, with a pole at $\psi(b)$ and with the same
positive limit at $\pm\infty$. Thus, it crosses exactly twice the increasing
exponential function $e^{2\pi (b-a)\lm}$.
\end{proof}

The next claim will enable us to bound $h^{a,b}_1$ from below, on most of the
real line. Denote by $z_1,z_2\in\R$ ($z_1<z_2$) the two real zeroes of
$h^{a,b}_1$ whose existence is guaranteed by Claim~\ref{clm: two zeroes}.  We
also use the notation $B(x,\delta)$ for the interval of radius $\delta>0$
around $x\in \R$.

\begin{clm}\label{clm: h1 bound}
For all $\delta>0$, there exists $c_\delta>0$ such that for all $\lm\in \R
\setminus \left(B(z_1,\delta)\cup B(z_2,\delta) \right)$:
\begin{equation*}
h^{a,b}_1(\lm)>c_\delta (1+\lm^2) \max (e^{2a \cdot 2\pi \lm}, e^{2b \cdot
2\pi\lm}).
\end{equation*}
\end{clm}

\begin{proof}
Since the function $\frac{h^{a,b}_1(\lm)}{(1+\lm^2)e^{2a\cdot 2\pi\lm}} =
\left( \frac {l^a_1(\lm) - \l^b_1(\lm)e^{2\pi(b-a)\lm }} {\sqrt{1+\lm^2}}
\right)^2 $ approaches strictly positive limits as $|\lm|\to \infty$, there
exist $M_a, c_a>0$ such that
$$\forall |\lm|\ge M_a: h^{a,b}_1(\lm) \ge c_a (1+\lm^2) e^{2a\cdot 2\pi\lm}. $$
Similarly, there exist some $M_b, c_b>0$ such that $\forall |\lm|\ge M_b:
h^{a,b}_1(\lm) \ge c_b(1+\lm^2)e^{2b\cdot 2\pi\lm}.$ Take $M=\max(M_a,M_b)$.
Since $h(\lm)$ attains a positive minimum on $[-M,M]\setminus
\left(B(z_1,\delta)\cup B(z_2,\delta)\right)$, there exists some $c>0$
such that for all $\lm$ in this set, $h(\lm)\ge c (1+\lm^2)\max(e^{2a \cdot
2\pi \lm},e^{2b \cdot 2\pi \lm} ).$ Choosing $c_\delta = \min(c,
c_a,c_b)$ will yield the result.
\end{proof}

The next claim is a slight modification of the previous one, in order to fit
our specific need. Denote
$\diag=\indep.$

\begin{clm}\label{clm: restrict rho}
For every $\delta>0$ there exist a set $F=F_\delta =\R \setminus (I_1\cup
I_2)$ such that $I_j$ is an interval containing $z_j$ and of length at most
$\delta$ ($j=1,2$), $\rho(F)>0$, and there exists $c_\delta>0$ such that for
all small enough $\ep$,
\begin{equation*}
h(\lm+\tau)\ge c_\delta (1+(\lm+\tau)^2)\max \left(e^{2a \cdot 2\pi (\lm+\tau)},
e^{2b \cdot 2\pi(\lm+\tau)}\right),
\end{equation*}
for all $\lm,\tau\in (F\times F)\cap \diag.$
\end{clm}

\begin{proof}
Choose $F=\R \setminus \left(B\left(\frac {z_1}{2}, \delta_0\right)\cup
B\left(\frac {z_2}{2}, \delta_0\right) \right),$
 where
$\delta_0\le \delta$ is small enough so that $\rho(F)>0$. Then, for $\ep \le
\delta_0$ and $(\lm, \tau)\in (F\times F) \cap \diag$, we have
$$|\lm+\tau-z_j|\ge |2\tau-z_j|-|\lm -\tau| \ge 2\delta_0-\ep \ge \delta_0.$$
Choosing the constant $c_\delta>0$ which is the consequence of applying
Claim~\ref{clm: h1 bound} will end our proof.
\end{proof}

\subsubsection{Convergence properties of the functions $S_k$}
Recall the definition of $S^{a,b}_k$ in~\eqref{eq: Sk1}.
We stress, once again, that $S^{a,a}_k(x)$ denotes the evaluation of the same mixed partial
derivative at the point $(x,a,a)$.
Our goal in this subsection is to prove the following Lemma.

\begin{lem}\label{lem: sum Sk}
If condition \eqref{eq: cond L2} is satisfied, then for every $k\in\N$ the
functions
  $S^{a,a}_k(x)$, $S^{a,b}_k(x)$ and $S^{b,b}_k(x)$ are in $L^1(\R)$
  with respect to the variable $x$. Moreover,
  $$\sum_{k\ge 1} \frac 1 {k^2} \int_\R S_k(x)dx \text{ converges},$$
with any of the three possible superscripts on the letter $S$.
\end{lem}

We will need some convergence properties of the function $q$ and its partial derivatives.
The first is simple.

\begin{obs}\label{obs: step 1}
 Let $g$ be one of the functions $q, q_a, q_b$ or $q_{ab}$. Then
$g(x,a,a)$, $g(x,a,b)$ and $g(x,b,b)$ are all in $\left(L^1\cap
L^\infty\right)(\R)$ with respect to the variable $x$.
\end{obs}
\begin{proof}
Follows from part 3 of Lemma~\ref{lem: q}.
\end{proof}

The next two claims require some more effort.

\begin{clm}\label{clm: step 2}
 The sum $\sum_{m\ge 1} \int_\R q^m q_a q_b \ dx$ converges.
\end{clm}
\begin{proof}
  We will show, in fact, that the positive series $\sum_{m\ge 1} \int_\R q^m |q_a q_b| \ dx$ converges.

  First, in case we are evaluating at $(x,a,b)$ ($a<b$), our series converges due to~\eqref{eq: int qab, qa qb} and the bound
in part 2 of Lemma~\ref{lem: q}.
  Now assume we are evaluating at $(x,t,t)$ (where $t\in\{a,b\}$). As we deal with a positive series, it is enough to show that both
  \begin{enumerate}[label=(\Roman{*})]
  \item\label{item: |x|<1} \quad $\sum_{m\ge 1} \int_{-1}^1 q^m |q_a q_b| \ dx
      < \infty$, and
  \item\label{item: |x|>1} \quad $\sum_{m\ge 1} \int_{|x|\ge 1} q^m |q_a q_b| \
      dx< \infty$.
  \end{enumerate}
  Denote by $C=\sup_{x\in\R} |q_a q_b(x,t,t)| \in (0,\infty)$. The sum in~\ref{item: |x|>1} is bounded by
  $$C \sum_{m \ge 1} \int_{|x|\ge 1} q^m(x,t,t) \ dx =C \int_{|x|\ge 1 } \frac {q}{1-q}(x,t,t)\ dx \le C^\prime \int_{\R} q(x,t,t) \ dx,$$
where $C^\prime\in (0,\infty)$ is another constant. $C$, $C'$ and $\int_{\R}
q(x,t,t) dx$ are all finite by part 3 of Lemma~\ref{lem: q}.

We turn to show~\ref{item: |x|<1}. By parts 1 and 4 of Lemma~\ref{lem: q},
the sum
\[
\sum_{m\ge 1}q^m |q_a q_b| \ dx = \frac{|q_a q_b|}{1-q}
\]
is well-defined for all $x$ (including $x=0$). By the monotone convergence
theorem, item~\ref{item: |x|<1} is then equivalent to
$$\int_{-1}^1 \frac{|q_a q_b|}{1-q} (x,t,t) \ dx < \infty, $$
which is indeed finite as an integral of a continuous function on $[-1,1]$.
\end{proof}

\begin{clm}\label{clm: step 3}
 The sum $\sum_{m\ge 1}\frac 1 {m+2} \int_\R q^m\ dx$ converges.
\end{clm}

\begin{proof}
We use a fact which appears in standard proofs of the Central Limit Theorem (CLT). For completeness, we include a proof in
appendix~\ref{sec: CLT}.

\begin{lem}\label{lem: CLT}
Let $g\in L^1(\R)$ be a probability density, i.e., $g\ge 0$ and $\int_\R
g=1$. Suppose further that
\begin{enumerate}[label={\rm(\alph{*})}]
\item\label{item: moment3}\quad $\int_{\R }|\lm|^k g(\lm) d\lm <\infty$ for
    $k=1,2,3$ and
\item\label{item: G integ} \quad  $\int_\R |\FF[g](x)|^\nu \ dx <\infty$ for
    some $\nu\ge 1$.
\end{enumerate}
Then there exists $C>0$ such that for all $m\ge \nu$,
$$\int_\R |\FF[g](x)|^{m} dx < \frac C {\sqrt m}. $$
\end{lem}

We would like to apply Lemma~\ref{lem: CLT} to
$$g^{a,b}(\lm)=\frac {e^{2\pi (a+b) \lm}p(\lm)}{r(ia+ib) } . $$
Notice that this is the density of a probability measure.
This choice also obeys
the extra integrability conditions in the lemma (Condition \ref{item: moment3} follows from the exponential moment assumption~\eqref{eq: cond
L1}, and condition~\ref{item: G integ} with $\nu=2$ follows from the $L^2$ assumption~\eqref{eq: cond L2}). We see now that
$$q(x,a,b)= \frac {r(ia+ib)^2}{r(2ia)r(2ib)} \cdot|\FF[g^{a,b}](x)|^2 \le  |\FF[g^{a,b}](x)|^2,  $$
the last inequality following from the log-convexity of $y\mapsto r(iy)$.
Similarly we define $g^{a,a}$ and have $q(x,a,a)= |\FF[g^{a,a}](x)|^2$. Thus
in all three cases of evaluation, using the lemma with the appropriate
function $g$ yields:
\begin{align*}
\sum_{m\ge 1}\frac 1 {m+2} \int_\R q^m\ dx
&\le \sum_{m\ge 1}\frac {1}{m+2}\int_\R |\FF[g](x)|^{2m} dx\\
&< C \sum_{m\ge 1}\frac 1{(m+2) \sqrt{2m}} < \infty,
\end{align*}
as required.
\end{proof}

\begin{proof}[Proof of Lemma \ref{lem: sum Sk}]

Taking derivative by the chain rule, we see that:
\begin{equation}\label{eq: Sk}
  S_k=\ptwo{a}{b} \left\{q^k\right\} =
  \begin{cases}
    q_{ab}, & k=1\\
    k(k-1)q^{k-2}q_a q_b, + kq^{k-1}q_{ab} & k>1.
  \end{cases}
\end{equation}
Now, the fact that $S^{a,a}_k$, $S^{a,b}_k$ and $S^{b,b}_k$ are in $L^1(\R)$ with respect to
$x$ follows from Observation~\ref{obs: step 1}.

We turn now to prove the ``moreover'' part of the claim.
We use~\eqref{eq: Sk} in order to rewrite the desired series:
\begin{align}\label{eq: sum int Sk}
&\sum_{k\ge 1}\frac 1{k^2} \int_\R S_k(x) dx  \notag \\
 &\quad =\int_\R q_{ab} \ dx +\sum_{k\ge 2} \int_\R q^{k-2} q_a q_b
\ dx + \sum_{k\ge 2}\frac 1 k \int_\R q^{k-2}\left(q q_{ab}-q_a q_b
\right)\ dx.
\end{align}
Once again, all functions are evaluated at $(x,a,a)$, $(x,a,b)$ or $(x,b,b)$
and what follows holds for each of the three options. By Observation~\ref{obs: step 1},
\begin{equation}\label{eq: int qab, qa qb}
\int_\R |q| \ dx < \infty,\quad \int_\R |q_{ab}| \ dx < \infty, \quad \int_\R |q_{a} q_b| \
dx<\infty,
\end{equation}
and in particular the left-most term in~\eqref{eq: sum int Sk} is finite.
For the middle sum in~\eqref{eq: sum int Sk}, convergence follows from Claim~\ref{clm: step 2} and~\eqref{eq: int qab, qa qb}.
Convergence of the right-most sum in~\eqref{eq: sum int Sk} follows from Claim~\ref{clm: step 3} and~\eqref{eq: int qab, qa qb}.
This ends the proof of Lemma~\ref{lem: sum Sk}.
\end{proof}


\subsection{Existence and Positivity.}\label{sec: lower bound}
In this section we prove that $L_1$ exists and belongs to $(0,\infty]$.
If $\rho$ has at least one atom, Theorem~\ref{thm: var quadratic} implies that
$\lim_{T\to\infty}\frac{\V }{T^2}>0$, and therefore $L_1=\lim_{T\to\infty}\frac {\V}{T} = \infty$. We thus assume that $\rho$ has no atoms.

Using the formula for the variance obtained in Proposition~\ref{prop: var
formula}, and recalling the functions $h^{a,b}_k$ are non-negative, we see
that the limit $L_1$ exists and is in $[0,\infty]$. More effort is needed in
order to establish that $L_1>0$. We begin with a simple bound arising from
Proposition~\ref{prop: var formula}:
\begin{align}\label{eq: for positivity}
\ds \liminf_{T\to\infty}\frac{V_f^{a,b}(T)}{2T} &=\frac
1{4\pi^2}\liminf_{T\to\infty} \sum_{k\ge 1 } \frac 1{ k^2}
\int_{\R}\int_\R T\sinc^2(2\pi T(\lm-\tau)) h^{a,b}_k(\lm+\tau)
d\rho^{*k}(\lm)d\rho^{*k}(\tau) \notag\\
&\ge \frac 1 {4\pi^2} \liminf_{T\to\infty }\int_\R\int_\R T\
\sinc^2(2\pi T(\lm-\tau)) \ h^{a,b}_1(\lm+\tau)\  d\rho(\lm)d\rho(\tau)\notag\\
&\ge C_0\ \liminf_{\ep\to 0+} \int_\R \int_\R \frac 1 {2\ep}\indep
h^{a,b}_1(\lm+\tau)\ d\rho(\lm)d\rho(\tau),
\end{align}
where $C_0>0$ is an absolute constant.
the first inequality follows from considering only the first term in the original non-negative series.
The second inequality follows by noticing that
$ T\sinc^2(2\pi T\, \lm) \ge \frac {4T}{\pi^2} \ind\{\lm: |\lm|<\frac 1 {4T}\}$.

Fix a parameter $\delta>0$, and fix $F=F_\delta$ to be the set provided by
Claim~\ref{clm: restrict rho}. Continuing~\eqref{eq: for
positivity}, we have
\begin{align}\label{eq: low}
\ds \liminf_{T\to\infty}\frac{V_f^{a,b}(T)}{2T} &\ge c_\delta
\liminf_{\ep \to 0} \iint_{F\times F}
\frac 1{2\ep}\ind_{\diag}(\lm,\tau)\ e^{2\pi \cdot 2a(\lm+\tau)} d\rho(\lm)d\rho(\tau) \nonumber\\
& = c_\delta \liminf_{\ep \to 0} \int_{F} \frac 1 {2\ep}
\rho_{2a}\left((\tau-\ep,\tau+\ep)\cap
F\right) d\rho_{2a}(\tau) \nonumber \\
&= c_\delta  \liminf_{\ep \to 0}\int_\R \frac 1 {2\ep}
\mu\left(\tau-\ep,\tau+\ep\right) d\mu(\tau),
\end{align}
where $\mu$ is the restriction of $\rho_{2a}$ to $F$, i.e. $\mu(\p)=
\rho_{2a}(\ind_F\cdot \p)$ for any test-function $\p$. Notice that by the
choice of $F$, $\mu(\R)=\rho_{2a}(F)>0$.
By Lemma~\ref{lem: mu on balls dmu}, the RHS of~\eqref{eq: low} is strictly positive, and we conclude that $L_1>0$.


\subsection{Linear Variance}\label{sec: lin finite}

Consider again the first line of \eqref{eq: sum out}.
Recall that, as we saw in Section~\ref{sec: deriv}, each term of the series in the
RHS of~\eqref{eq: sum out} is non-negative. Therefore, by the monotone convergence theorem:
\begin{align*}
\lim_{T\to\infty} \frac{\V}{2T} =
 \frac 1 {8\pi^2} \sum_{k\ge 1}\frac 1{k^2}\lim_{T\to\infty}
  \int_{-2T}^{2T} \left(1-\frac {|x|}{2T}\right)\
(S^{a,a}_k(x)-2S^{a,b}_k(x)+S^{b,b}_k(x)) \ dx,
\end{align*}
provided that the limit of each term on the RHS exists. These limits can be computed using Observation \ref{obs: lim triangle}:
\begin{align*}
  \lim_{T\to\infty} \frac{\V}{2T} =
  \frac 1 {8\pi^2} \sum_{k\ge 1}\frac 1{k^2}
  \int_\R
(S^{a,a}_k(x)-2S^{a,b}_k(x)+S^{b,b}_k(x)) \ dx,
\end{align*}
which is finite by Lemma~\ref{lem: sum Sk}.

Lastly, we explain how to obtain the form of $L_1$ appearing in Remark~\ref{rmk: final L1}.
By monotone convergence theorem, we may take term-by-term limit as
$T\to\infty$ in Proposition~\ref{prop: var formula}, and get:
\begin{equation*}
\ds \lim_{T\to\infty}\frac{V_f^{a,b}(T)}{2T} =\frac 1 {8\pi^3}
\sum_{k\ge1} \frac {1}{k^2}\int_\R \left(p^{*k}(\lm)\right)^2
h^{a,b}_k(2\lm)d\lm \in (0,\infty).
\end{equation*}

\section{Theorem~\ref{thm: var super}: Super-linear variance}\label{sec: super}
In this section we prove the two items of Theorem~\ref{thm: var super}, in
reverse order.

\subsection{Item \ref{item: particular a,b}: Super-linear variance for particular $a,b$}
Assume condition~\eqref{eq: cond inf} holds for the particular $a$ and $b$ at
hand. Fix a parameter $\delta>0$, and let $F=F_\delta$ be the set provided by
Claim~\ref{clm: restrict rho}. The premise ensures that, if $\delta$ is small
enough, at least one of the measures $(1+\lm)\rho_{2a}|_{F_\delta}$ and
$(1+\lm)\rho_{2b}|_{F_\delta}$ does not have $L^2$-density. WLOG assume it is
the former. At first, assume also $\rho_{2a}|_F$ is not in $L^2$. Repeating
the arguments of the Subsection~\ref{sec: lower bound} we get the lower bound
$$ \liminf_{T\to\infty}\frac{V_f^{a,b}(T)}{2T} \ge c_\delta \int_\R |\FF[\mu]|^2(x) dx,$$
where $\mu=\rho_{2a}|_F$ and $c_\delta>0$. The LHS is therefore infinite, and
so $L_1=\infty$.

We are left with the case that $\lm \rho_{2a}|_{F_\delta}$ does not have
$L^2$-density, but $\rho_{2a}|_{F_\delta}$ does (denote it by $p_{2a}$). The
argument is similar. Continuing from \eqref{eq: for positivity} and employing
Claim~\ref{clm: restrict rho}, we get
\begin{align*}
\liminf_{T\to\infty} \frac{V_f^{a,b}(T)}{2T}
& \ge c_\delta \liminf_{\ep \to 0} \int_F \int_F \frac 1 {2\ep}
\ind_{(\tau-\ep,\tau+\ep )}(\lm) (\lm+\tau)^2 p_{2a}(\lm) \ p_{2a}(\tau) d\lm d\tau\\
&\ge  c_\delta\cdot 4\int_K \lm^2 p_{2a}(\lm)^2 d\lm,
\end{align*}
where $K\subset F$ is compact. But, by our assumption, by choosing $K$
properly the last bound can be made arbitrarily large, so that
$\lim_{T\to\infty}\frac{V_f^{a,b}(T)}{2T}=\infty$.

\subsection{Item \ref{item: all a,b}: Super-linear variance for almost all $a,b$}
Let $\rho$ be such that the condition in item \ref{item: all a,b} holds. If
$\rho$ has a singular component, then the condition in item \ref{item:
particular a,b} holds for all possible $a,b$ and so $L_1(a,b)=\infty$ with no
exceptions. Otherwise, $\rho$ has density $p(\lm)$. Define the set
$$E = \{(a,b):\ a,b\in J, \ a<b, \ \text{the condition in item~\ref{item: particular a,b} fails for } a, b\}. $$
If $E=\emptyset$, once again $L_1(a,b)=\infty$ for all $a,b\in J$ with no
exceptions.

Assume then there is some $(a_0,b_0)\in E$. This means there exists
$\lm_1,\lm_2$ such that for any pair of intervals $I_1,I_2$ such that
$\lm_j\in I_j$ ($j=1,2$), both the functions $(1+\lm^2)e^{2\pi \cdot 2a_0
\lm}p(\lm)$ and $(1+\lm^2)e^{2\pi\cdot 2b_0\lm}p(\lm)$ are in $L^2(\R
\setminus (I_1\cup I_2))$, but at least one of them (WLOG, the former) is not
in $L^2(\R)$. Observe that the existence of such $\lm_1,\lm_2$ depends solely
on $p(\lm)$, and may therefore be regarded as independent of the point
$(a_0,b_0)\in E$. Moreover, at least one among $\lm_1$ and $\lm_2$ (say,
$\lm_1$) is such that for any neighborhood $I$ containing it, $p\not\in
L^2(I)$.

Suppose now $a,b\in E$ are such that
\begin{equation}
\label{eq: hab(lm)>0} h^{a,b}_1(\lm_1)>0,
\end{equation}
where $h^{a,b}_1(\lm)=\left(l^a_1(\lm)e^{2\pi a\lm}-l^b_1(\lm)e^{2\pi
b\lm}\right)^2$ is the function appearing in the the first term of our
asymptotic formula, and in the lower bound in inequality~\eqref{eq: for
positivity}. Recall $h^{a,b}_1$ is non-negative and has only two zeroes by
Claim~\ref{clm: two zeroes}.

We may choose $\delta>0$ smaller than the minimal distance between $\lm_1$
and a zero of $h^{a,b}_1$, and then construct $F=F_\delta$ as in
Claim~\ref{clm: restrict rho}. Certainly $\lm_1\in F_\delta$, and so the
measure $\mu = \rho_{2a}|_{F_\delta}$ is not in $L^2(\R)$ (it is even not in
$L^2(I)$ for any neighborhood $I$ of $\lm_1$). Just as in
subsection~\ref{sec: lower bound} we shall get
$$ \liminf_{T\to\infty} \frac{\V}{2T} \ge c_\delta \int_\R |\FF[\mu]|^2(x) dx=\infty. $$

We end by showing that for a given point $\lm_1\in\R$ and a given $a\in J$,
the set of $b \in J$ which do not obey~\eqref{eq: hab(lm)>0} is finite.
Indeed, this is the set
$$\{b\in J: \ h^{a,b}(\lm_1)=0\}=\{b\in J:\ \p(a)=\p(b)\}$$
where
$$\p(y)= e^{2\pi
y\lm_1 }l^y_1(\lm_1)=\pone{y}\left(\frac{e^{2\pi \lm_1 y}}{r(2iy)}\right).$$
Suppose the desired set is not finite. Since $\p$ is real-analytic, it must
be constant on $J$. But then $r(2iy)=\frac{e^{2\pi \lm_1 y}}{cy+d}$ for some
$c,d\in\R$, and the corresponding spectral density would satisfy
condition~\eqref{eq: cond L2} for all relevant $a$, $b$. This contradiction
ends the proof.


\appendix
\section{Uniform exponential decay of tails}\label{sec: off}
In this appendix we prove Proposition~\ref{prop: off}. We follow closely~\cite[Proposition 5.1]{Naomi} and \cite[Ch. 7]{GAF book} which prove similar concentration bounds, known as ``Offord-type estimates''.
We rely on the following lemma, which follows either from \cite[Lemma 6.1]{Naomi} or \cite[Lemma 7.1.2]{GAF book}.
\begin{lem}\label{lem log bound}
If $\eta \sim \mathcal N_{\C} (0,\sigma^2)$, and $E$ is an event in the
probability space with $\Pro(E)=p$, then:
$$ |\E (\chi_E\log|\eta|)|  \leq p \left[\frac p {2}-2 \log p+\log \sigma \right].$$
\end{lem}

\begin{proof}[Proof of Proposition~\ref{prop: off}]
Take $\phi(z) =
\phi_T(z)$ to be a real $C^2$ function, whose support is \\ $[-\frac 1 2-T,T+\frac 1
2]\times[a^\prime,b^\prime]$ with $-\Delta<a^\prime<a<b<b^\prime<\Delta$, and
which takes the value $1$ on $R_T=[-T,T]\times[a,b)$. We may build such
$\phi_T(z)$ that will obey also the bound $\norm{\Delta\phi}_{L^1} <
10(T+b-a)$. Fix $s>0$. We are interested in dominating the probability of the event $A_T
=\{n_f(R_T) > sT \} $. Write $p = p_T = \Pro(A_T)$.

We have
\[
n_f(R_T) < \frac 1 {2\pi}\int \Delta\phi_T(z) \log|f(z)| dm(z)\,,
\]

and therefore,
\begin{align*}
sT \cdot p &\leq \E (\chi_{A_T} n_f(R_T)) \leq \E \left( \chi_{A_T}
\frac 1 {2\pi}\int \Delta\phi(z) \log|f(z)| dm(z)  \right) \\
&= \frac 1 {2\pi} \int \Delta\phi \: \E \left(\chi_{A_T} \log|f(z)| \right) dm(z) \\
&\leq \frac {1}{2\pi} \norm{\Delta \phi}_{L^1} \sup_{z\in\R\times[a^\prime,b^\prime]} \E
\left(\chi_{A_T} \log|f(z)| \right)
\end{align*}

The exchange of expectation and integral is justified by Fubini's
theorem, as follows:
\begin{align*}
&\int_D \E \big|\Delta\phi(z) \cdot \log |f(z)|\,\big| dm(z)  = \int_D |\Delta\phi(z) |\, \E\big|\log |f(z)| \,\big|
<\sup \E\big| \log|f(z)| \big| \cdot \norm{\Delta\phi}_{L^1}<\infty.
\end{align*}

Applying lemma \ref{lem log bound} with $\eta =f(z)$, we get:
$$ \sup_{z\in \R\times[a^\prime,b^\prime]} \E (\chi_{A_T} \log|f(z)|) < p(\alpha -2 \log p).$$
for some constant $\alpha>0$ (depending on $\sup_{y\in[a',b']} \E |f(iy)|^2$). Putting all this
together, we get:
$$sT\cdot p \leq \frac 5 \pi (T+b-a) p (c_1 - 2 \log p ) ,$$
which leads to the exponential bound we strived for:
$$\exists c,C>0 \:\text{such that } p_T = \Pro(n_f(R_T) > T\,s) \leq C e^{-c s }, \: \forall T\geq 1\,. $$
\end{proof}

\section{Justification for changing operations}\label{sec: tec}

In this section we prove Lemmas~\ref{lem: I,II} and \ref{lem: deriv}.
\subsection{Proof of Lemma~\ref{lem: I,II}}
We begin by showing~\ref{item: one dim}.
Let $1<p<2$ be arbitrary, and let $q>2$ be such that $\frac 1 p+\frac 1 q=1$. H\"{o}lder's inequality implies:
\begin{align*}
\int_{-T}^T \E\left| \frac{f'(t+ia)}{f(t+ia)}\right| &\le \int_{-T}^T \E [|f'(t+ia)|^q]^{1/q} \E[|f(t+ia)|^{-p}]^{1/p} dt\\
& \le \E [|f'(ia)|^q]^{1/q} \E[|f(ia)|^{-p}]^{1/p}  \cdot T  < \infty,
\end{align*}
where finiteness follows from $f'(ia)$ and $f(ia)$ being complex Gaussian random variables, thus having finite moments of any order.

We now turn to prove~\ref{item: two dim}.
We use the notation $f\lesssim g$ to stand for the inequality $f\le C\cdot g$, where $C>0$ is a constant (which may vary from line to line).
Similarly, $f\eqsim g$ stands for $f=C\cdot g$ with some $C>0$.

 As before, let $1<p<2$ and take $q>2$ to obey $\frac 1 p +\frac 1 q =1$ . By H\"{o}lder inequality we have
\begin{align} \label{eq: Holder}
&\intt \intt \E  \bigg| \frac{f'(t+ia) \ f'(s+ib)} {f(t+ia) \ f(s+ib)}\bigg|  dt ds \notag \\
&\le \intt \intt \left[ \E   \bigg| f'(t+ia) f'(s+ib) \bigg|^{q} \right]^{1/q} \  \left[\E  \bigg| f(t+ia) f(s+ib) \bigg|^{-p} \right]^{1/p} dt \ ds \notag \\
& \lesssim \intt \intt \E \left[  \bigg| f(t+ia) f(s+ib) \bigg|^{-p} \right]^{1/p} dt\  ds.
\end{align}
The last inequality is an application of Cauchy-Schwarz inequality and stationarity, as follows:
\begin{align*}
 &\left[ \E  \bigg| f'(t+ia) f'(s+ib) \bigg|^{q} \right] ^{\frac 1 q}
 \\
 &\le \left( \sqrt{\E |f'(t+ia)|^{2q} \ \E |f'(s+ib)|^{2q}} \right)^{\frac 1 q}
 = \left(\E [ |f'(ia)|^{2q} ] \E [ |f'(ib)|^{2q} ] \right) ^{\frac 1 {2q}} < \infty,\notag
\end{align*}
where again finiteness follows from $f'(ia)$ and $f'(ib)$ being Gaussian.

Let
\[
A = \left\{(t,s)\in [-T,T]^2: \ | r(t-s+ia+ib)|^2 \le \frac 2 3 r(2ia) r(2ib). \right\}.
\]
We split the last integral in~\eqref{eq: Holder} into two parts: on $A$ and on $A^c = [-T,T]^2 \setminus A$.
For the integral on $A$ we use the following lemma:

\begin{lem}\label{lem: for small r}
Suppose $\xi_1$, $\xi_2$ are independent $\calN_\C(0,1)$ random variables, and let $Z_1 = \alpha \xi_1$ and $Z_2 = \beta \xi_1 +\gamma \xi_2$
where $\alpha, \beta, \g \in \C\setminus\{0\}$. Let $1<p<2$.
Then
\[
\E \left[ \ |Z_1 Z_2|^{-p}\ \right] \le |\alpha \gamma|^{-p}\ \Gamma \left( 1-\frac {p}{2} \right)^2
\]
\end{lem}
We include the proof in Section~\ref{sec: lems}.
We apply Lemma~\ref{lem: for small r} with $Z_1=f(t+ia)$ and $Z_2 = f(s+ib)$, which yields the choice of parameters $\al, \beta, \g$ so that
\begin{align}\label{eq: choice}
&\al=\sqrt{r(2ia)}, \quad
\alpha \ov{\beta} = \ov{r(t-s+ia+ib)},\quad
|\beta|^2+|\gamma|^2 =r(2ib).
\end{align}
In particular,
\[
| \alpha \gamma| = \sqrt{r(2ia) r(2ib) -| r(t-s+ia+ib)|^2 } .
\]
Using this in Lemma~\ref{lem: for small r}, we have:
\begin{align*}
\iint_A & \left( \E \bigg| f(t+ia) f(s+ib) \bigg|^{-p} \right)^{1/p} dt\  ds \\
& \lesssim \iint_A\left(r(2ia) r(2ib) -| r(t-s+ia+ib)|^2 \right)^{-p/2} dt \ ds,
\end{align*}
which is bounded by the definition of $A$.
In order to bound the integration on $A^c$, we use another lemma (which is also proved in Section~\ref{sec: lems}).

\begin{lem}\label{lem: for big r}
Suppose $\xi_1$, $\xi_2$ are independent $\calN_\C(0,1)$ random variables, and let $Z_1 = \alpha \xi_1$ and $Z_2 = \beta \xi_1 +\gamma \xi_2$
where $\alpha, \beta, \g \in \C\setminus\{0\}$. Suppose $M>0$ is such that $|\frac{\gamma}{\beta}| < M$, and let $1<p<2$. Then there exists a constant $c>0$, depending only on $M$ and $p$, such that
\[
\E \left[ \frac 1 {| Z_1 Z_2 |^p }\right] \le \frac {c} {|\alpha \beta|^p} \left|\frac{\gamma}{\beta}\right|^{2-2p}.
\]
\end{lem}

We apply this lemma again to $Z_1=f(t+ia)$ and $Z_2=f(s+ib)$, so the choice of parameters in~\eqref{eq: choice} remains valid. Thus,
\[
\left| \frac {\g}{\beta}\right|^2 = \frac{ r(2ia)r(2ib)}{|r(t-s+ia+ib)|^2}-1
\]
is uniformly bounded for $(t,s)\in A^c$.
Applying Lemma~\ref{lem: for big r} we get that for some $c>0$,
\begin{align}\label{eq: Ac}
 \iint_{A^c}  &\left( \E \bigg| f(t+ia) f(s+ib) \bigg|^{-p} \right)^{1/p}  dt\  ds\notag  \\
&\lesssim \iint_{A^c} \left( \frac c {|r(t-s+ia+ib)|^p} \ \frac {(r(2ia)r(2ib)-|r(t-s+ia+ib)|^2)^{1-p} }{|r(t-s+ia+ib)|^{2-2p}} \right)^{\frac 1 p} dt \ ds   \notag \\
&\lesssim \iint_{A^c} \left(r(2ia)r(2ib)-|r(t-s+ia+ib)|^2 \right)^{-\frac{p-1}{p} } dt\ ds \\
& \lesssim \int_{ \tilde A^c} \left(r(2ia)r(2ib)-|r(x+ia+ib)|^2 \right)^{-\frac{p-1}{p} } dx \notag,
\end{align}
where $\tilde A^c$ is the one-dimensional set
\[
\tilde A^c = \{x\in [-T,T]:\: |r(x+ia+ib)|^2 > \frac 2 3 r(2ia) r(2ib)\} .
\]
The last inequality in~\eqref{eq: Ac} is obtained by a simple change of varibales. One step before that in~\eqref{eq: Ac}
we bounded $|r(t-s+ia+ib)|^{-1}$ from above by a constant, using the definition of $A^c$.

Before continuing, we notice that
\begin{align}\label{eq: CS}
|r(x+ia+ib)|^2 &=\left|\int_\R e^{-2\pi i x\lm} e^{2\pi\cdot (a+b) \lm} d\rho(\lm)\right|^2  \notag\\
& \le \left(\int_\R e^{2\pi\cdot (a+b) \lm} d\rho(\lm)\right)^2   \quad ( = r(ia+ib)^2  ) \notag \\
& \le \int_\R e^{2\pi\cdot 2a \lm} d\rho(\lm) \cdot \int_\R e^{2\pi\cdot 2a \lm} d\rho(\lm) \notag \\
&  = r(2ia) r(2ib) ,
\end{align}
and the inequality is sharp when $a\ne b$ (see also part 2 of Lemma~\ref{lem: q} below).
Therefore, if $a\ne b$, the  last integral in~\eqref{eq: Ac} is finite.

In case $a=b$, there may be only a finite number of isolated points $x_0$ for which $|r(x_0+2ia)|^2=r(2ia)^2$. Taylor expansion near any of those points gives $|r(x+2ia)|^2= r(2ia)^2-C(x-x_0)^2 + o((x-x_0)^2)$ as $x$ tends to $x_0$ (here $C\ge 0$ since $|r(x+2ia)|^2 \le r(2ia)^2$ by taking $a=b$ in~\eqref{eq: CS}).
So, in this case the finiteness of the integral~\eqref{eq: Ac} is equivalent to that of
$$\int_{|x-x_0|<\delta} (x-x_0)^{-2(p-1)/p} dx$$
(with some $\delta>0$), which is
indeed finite for $1<p<2$.

\subsection{Proof of Lemma~\ref{lem: deriv}}
We will justify in detail the second item, as the proof of the first is similar and simpler.

Fix $(t,s)\in[-T,T]^2$. We may assume that $|r(t-s +ia+ib)|^2< {r(2ia)r(2ib)}$ as the set of $(t,s)\in[-T,T]^2$
where this inequality does not hold is of measure zero (see discussion following~\eqref{eq: CS}).
Consider the random variables
\begin{align*}
U &(h_1,h_2) =\frac {\log |f(t+ia+ ih_1)|-\log |f(t+ia)|}{h_1}
\cdot \frac {\log |f(s+ib+ ih_2)|-\log |f(s+ib)|}{h_2}.
\end{align*}
These are well-defined for $0<h_1<\delta_1$ and $0<h_2<\delta_2$, where $\delta_1,\delta_2$ are properly chosen numbers (we use here that almost surely, there are no zeroes on the vertical lines $\{t\}\times [a,a+\delta_1]$ and $\{s\}\times [b,b+\delta_2]$, as was explained in Section~\ref{sec: significant edges}).
Notice that almost surely,
\[
\lim_{h_1\to 0+} \lim_{h_2\to 0+} U(h_1,h_2) = \ptwo{a}{b} [ \ \log |f(t+ia)| \log |f(s+ib)|\  ].
\]
Our goal \ref{eq: j2} can be understood as convergence in $L^1(\Pro)$ of the above limit. This will follow if the family
$\{  U(h_1,h_2) \}_{h_1,h_2}$ is uniformly integrable, i.e., if for every $\ep>0$ there exists $k>0$ such that for all relevant $h_1, h_2$ it holds that
\begin{equation*}
\E \left( |U(h_1,h_2)|\ind_{\{ |U(h_1,h_2)|\ge k \}}  \right)  < \ep.
\end{equation*}
Uniform integrability, in turn, would follow from the following statement:
\footnote{Indeed, suppose~\eqref{eq: uniform p bound} holds. Denoting by $q$ the number such that $\frac 1 p + \frac 1 q =1$, we apply H\"{o}lder's inequality  to get:
\[
\E\big(  |U(h_1,h_2)|\ind\{ |U(h_1,h_2)| \ge k \}  \big) \le \left( \E |U(h_1,h_2)|^p \right)^{1/p} \Pro\left(|U(h_1,h_2)|\ge k \right)^{1/q}
\lesssim \frac 1{k^{p/q}},
\]
so the definition of uniform integrability is satisfied.
}

\begin{equation}\label{eq: uniform p bound}
\exists p>1: \: \sup_{h_1,h_2} \E |U(h_1,h_2)|^p < \infty.
\end{equation}

Applying the Newton-Leibniz formula, the bound~\eqref{eq: d log} and Jensen's inequality we get:
\begin{align*}
|U(h_1,h_2)| &= \left| \frac 1 {h_1 h_2} \int_0^{h_2} \int_0^{h_1} \pone{y_1}\log |f(t+ia+iy_1)| \pone{y_2} \log|f(s+ib+iy_2)|  dy_1 \, dy_2 \right|  \\
&\le \int_{0}^{h_2} \int_0^{h_1} \left|  \frac{f'}{f}(t+ia+iy_1) \right| \, \left|  \frac{f'}{f}(s+ib+iy_1) \right|\frac { dy_1}{h_1} \, \frac {dy_2}{h_2} \\
& \le\left( \int_{0}^{h_2} \int_0^{h_1} \left|  \frac{f'}{f}(t+ia+iy_1) \right|^p \, \left|  \frac{f'}{f}(s+ib+iy_1) \right|^p\frac { dy_1}{h_1} \, \frac {dy_2}{h_2}\right)^{1/p }.
\end{align*}
Taking $1<p'<\frac 2 p$ and $q'$ such that $\frac 1 {p'} + \frac 1 {q'} = 1$, we apply H\"{o}lder's inqeuality to bound the last expression by
\begin{align*}
& \left( \int_0^{h_2}\int_0^{h_1} \E |f'(t+ia+iy_1) \ f'(s+ib+iy_2)|^{pq'}\,  \frac { dy_1}{h_1} \ \frac {dy_2}{h_2} \right)^{\frac 1{pq'}} \\
&\quad \times \left( \int_0^{h_2}\int_0^{h_1} \E |f(t+ia+iy_1) \ f(s+ib+iy_2)|^{-pp'} \, \frac { dy_1}{h_1} \ \frac {dy_2}{h_2} \right) ^{\frac 1{pp'} }
\end{align*}
Using Cauchy-Schwarz, the first integral is bounded by
\[
\left( \max_{y_1\in [0,\delta_1]} \E |f'(t+ia+iy_1)|^{2pq'}\max_{y_2\in [0,\delta_2]} \E |f'(s+ib+iy_2)|^{2pq'}   \right)^{\frac 1 {2pq'}},
\]
which is finite and independent of $h_1$ and $h_2$. Applying Lemma~\ref{lem: for small r} with the same choice of parameters as before, we may bound the second integral (up to a constant factor depending on $p$ and $p'$) by:
\begin{align*}
&\left( \int_0^{h_1} \int_0^{h_2} \left\{ r(2ia+2iy_1) r(2ib+ 2iy_2)-| r(t-s+ia+ib+iy_2+iy_2) |^2 \right\} ^{-\frac {pp'} 2} \frac{dy_1 }{h_1} \frac {dy_2}{h_2} \right)^{\frac 1{pp'} }\\
&\lesssim  \max_{y_1\in[0,\delta_1], y_2\in[0,\delta_2]} \left\{ r(2ia+2iy_1) r(2ib+ 2iy_2)-| r(t-s+ia+ib+iy_2+iy_2) |^2 \right\} ^{-\frac 1 2},
\end{align*}
which is again finite and independent of $h_1$ and $h_2$.
Our proof is complete.

\subsection{Proofs of auxiliary Lemmas}\label{sec: lems}
This part is dedicated to prove Lemmas~\ref{lem: for small r} and~\ref{lem: for big r} that were used earlier in this section.
\begin{proof}[Proof of Lemma~\ref{lem: for small r}]
Using the notations in the statement of the Lemma, we have:
\[
\E\left[|Z_1 Z_2|^{-p} \right]
= \frac 1 {\pi^2} \iint_{\C^2} | \alpha \xi_1 ( \beta \xi_1 + \g \xi_2) |^{-p} e^{-|\xi_1|^2-|\xi_2|^2} dm(\xi_1) dm(\xi_2)
\]
Now, by the Hardy-Littlewood re-arrangement inequality, we have:
\begin{align*}
\frac 1 \pi & \int_\C |\beta \xi_1 + \g \xi_2|^{-p} e^{-|\xi_2|^2} dm(\xi_2)
\le |\g|^{-p}\cdot \frac 1 \pi \int_\C |\xi_2|^{-p} e^{-|\xi_2|^2} dm(\xi_2) \\
& = |\g|^{-p} \ \Gamma\left(1-\frac p 2\right).
\end{align*}

So,
\begin{align*}
\E\left[ |Z_1 Z_2|^{-p} \right] &\le |\al \g|^{-p} \ \Gamma\left(1-\frac p 2\right) \cdot \frac 1 \pi \int_\C |\xi_1|^{-p} e^{-|\xi_1|^2} dm(\xi_1) \\
&= |\al \g|^{-p}\ \Gamma\left(1-\frac p 2\right)^2.
\end{align*}

\end{proof}

\begin{proof}[Proof of Lemma~\ref{lem: for big r}]
In this proof, the constant hidden by the ``$\lesssim$'' and ``$\eqsim$'' notation depends only on $M$ and $p$.

We begin by writing-out the desired expectation explicitly.

\begin{align}\label{eq: start}
\E \left[\ |Z_1 Z_2|^{-p}\ \right] &= |\al \beta|^{-p}\ \E\left[\  \big|\xi_1^2+\frac \g \beta \xi_1 \xi_2\big|^{-p} \ \right]\notag \\
& = |\al \beta|^{-p} \cdot \frac 1{\pi^2} \iint_{\C^2}
\big|z^2 + \frac \g \beta z w \big|^{-p} e^{-|z|^2-|w|^2}\ dm(z) \ dm(w) \notag \\
& = |\al \beta|^{-p} \pi^{-2} \int_\C |z|^{-p}
\left( \int_\C \big|z+\frac \g \beta w\big|^{-p} e^{-|w|^2} dm(w) \right)e^{-|z|^2} dm(z).
\end{align}

We bound the inner integral as follows:
\begin{align*}
& \int_\C   \big|z+\frac \g \beta w\big|^{-p} e^{-|w|^2} dm(w) \\
& \lesssim \int_{\substack{|w|\le \frac 1 2 \big| \frac{\beta}{\g} z \big|}  } |z|^{-p} e^{-|w|^2} dm(w)
+ |z|^{-p}  e^{-\frac 1 4 \big| \frac{\beta}{\g} z \big|^2} \left| \frac {\beta} {\g} z \right|^2
 +\int_{\substack{|w|> 2\big| \frac{\beta}{\g} z \big| } } \bigg| \frac{\g}{\beta} w\bigg|^{-p} e^{-|w|^2} dm(w) \\
& \eqsim |z|^{-p} \left(1- e^{-\frac 1 4 \big| \frac{\beta}{\g} z \big|^2}\right)
+\bigg|\frac \beta \g \bigg|^{2}|z|^{2-p}  e^{-\frac 1 4 \big| \frac{\beta}{\g} z \big|^2}
+\bigg| \frac \beta \g \bigg|^{p} I\left(\bigg|\frac {\beta}{\g} z\bigg| \right),
\end{align*}
where
\[
I\left(s\right) = \int_{|w|>2s} |w|^{-p} e^{-|w|^2} dm(w) \lesssim
\begin{cases}
1, & 0<s\le 1, \\
s^{-p}e^{-4s^2}, & s>1.
\end{cases}
\]

 The last bound is achieved by changing to polar coordinates, as follows:
\[
I(s) \eqsim \int_{2s}^\infty r^{-p+1} e^{-r^2} dr \lesssim s^{-p+1}  \int_{2s}^\infty e^{-r^2} dr \le
s^{-p+1} \frac 1 {2s} e^{-4s^2}.
\]

Returning to the double integral in~\eqref{eq: start}, we have:
\begin{align*}
&\E \left[ \big|\xi_1^2 + \frac \g \beta \xi_1 \xi_2 \big|^{-p} \right] \\
& \lesssim \int_\C \left\{ |z|^{-2p}\left(1-e^{-\frac 1 4 \big|\frac \beta \g z\big|^2 } \right)
+ \bigg|\frac \beta \g \bigg|^{2} |z|^{2-2p}e^{-\frac 1 4 \big| \frac \beta \g z \big|^2 }
+|z|^{-p} \left|\frac \beta \g \right|^p I\left( \bigg|\frac \beta \g z \bigg| \right) \right\} e^{-|z|^2} dm(z)
\end{align*}

This is the sum of three integrals, which we bound separately. For the first, we have:
\begin{align*}
\int_\C  &|z|^{-2p} \left(1-e^{-\frac 1 4 \big| \frac \beta \g z \big|^2} \right) e^{-|z|^2} dm(z) \\
&\lesssim \int_{|z|\le \big|\frac \g \beta \big|} \left|\frac \beta \g \right|^2 |z|^{2-2p} dm(z)
+\int_{|z|>\big|\frac \g \beta \big|} |z|^{-2p} e^{-|z|^2} dm(z) \\
& \eqsim \left|\frac \beta \g \right|^2 \left|\frac \g \beta \right|^{4-2p} + O(1) \eqsim \left|\frac \g \beta \right|^{2-2p}
\end{align*}

Denote $A=1+\frac 1 4 \big| \frac \beta \g \big|^2$.
Before estimating the second integral, we compute
\begin{align*}
\int_\C & |z|^{2-2p} e^{-A|z|^2 }dm(z)  \quad &  \\
&\eqsim\int_0^\infty r^{2-2p}e^{-A r^2} r dr \quad & [r = |z| \ ] \\
& = \frac 1 {2A} \int_0^\infty \left(\frac s A\right)^{1-p} e^{-s} ds \quad & [s = A r^2] \\
&\eqsim A^{-(2-p)}. &
\end{align*}
Thus, the second integral is
\begin{align*}
\left|\frac \beta \g\right|^{2} &\int_\C |z|^{2-2p} e^{-\left(1+\frac 1 4 \big|\frac \beta \g \big|^2\right)|z|^2 }dm(z)
\eqsim \bigg| \frac \beta \g \bigg|^{2} \left(1+\frac 1 4 \bigg| \frac \beta \g \bigg|^2 \right)^{-(2-p)}
\eqsim \left| \frac \g \beta \right|^{2-2p}.
\end{align*}

For the third integral we first compute
\begin{align*}
\int_{|z| > |\frac \g \beta|} & |z|^{-2p} e^{-A|z|^2 }dm(z)  \quad &  \\
&\eqsim\int_{|\frac \g  \beta|}^\infty r^{-2p}e^{-A r^2} r dr \quad & [r = |z| \ ] \\
& = \frac 1 {2A} \int_{A|\frac \g \beta|}^\infty \left(\frac s A\right)^{-p} e^{-s} ds \quad & [s = A r^2] \\
&\lesssim A^{p-1} \int_{1/ 4}^\infty  s^{-p}e^{-s} ds \eqsim A^{p-1}.&
\end{align*}

Finally, the third integral is
\begin{align*}
\bigg| & \frac \beta \g\bigg|^p \int_\C |z|^{-p}\, I\left(\bigg|\frac{\beta}{\g}z\bigg|\right) e^{-|z|^2} dm(z) \\
&\eqsim \bigg| \frac \beta \g \bigg|^p \left\{\int_{|z|<|\frac \g \beta|} |z|^{-p} e^{-|z|^2} dm(z)
+\bigg|\frac \beta \g \bigg|^{-p} \int_{|z|> |\frac \g \beta| }|z|^{-2p} e^{-(1+4\big|\frac \beta \g\big|^2)|z|^2} dm(z) \right\} \\
&\eqsim \bigg|\frac \beta \g\bigg|^p \ \bigg| \frac \g \beta \bigg|^{2-p}  
+ \left(1+4 \bigg| \frac \beta \g \bigg|^2 \right)^{p-1} \eqsim \bigg| \frac \g \beta \bigg|^{2-2p}.
\end{align*}
The proof is complete.

\end{proof}

\section{Moments of the characteristic function}\label{sec: CLT}
Here we prove Lemma~\ref{lem: CLT}, which estimates moments of the characteristic function (or Fourier transform) of a probability distribution.
We adapt the proof of the Central Limit Theorem appearing in~\cite[Ch. XV.5]{Fe}.
\begin{proof}[Proof of Lemma~\ref{lem: CLT}]
Write $G(x)=\FF[g](x)$. We may assume that $\int_\R \lm g(\lm)=0$ (otherwise
we shall consider, instead of $g$, the function $g_{\mu}(\lm)=g(\lm+\mu)$
where $\mu:=\int_\R \lm g(\lm)d\lm$. There is no penalty since
$|\FF[g_\mu](x)|=|\FF[g](x)|$ for all $x\in\R$). By assumption~\ref{item:
moment3}, $G(x)$ is thrice differentiable, and by the above assumptions
$G(0)=1$ and $G'(0)=0$.

To prove the lemma, it is enough to show that
$$\lim_{m\to \infty} \sqrt{m} \int_{\R} |G(x)|^m dx \text{ exists and is finite.}$$
Notice that $ \sqrt{m} \int_{\R} |G(x)|^m dx = \int_\R |G\left(x/{\sqrt
m}\right)|^m dx,$ and so it is enough to show that
\begin{equation}\label{eq: CLT goal}
\lim_{m\to \infty }\int_{\R } \left|\ \left|G\left(\frac{x}{\sqrt
m}\right)\right|^m - e^{-\frac{\alpha x^2}{2}}   \right| \ dx = 0,
\end{equation}
for some value of $\alpha>0$, which in fact is $\alpha:=G''(0)$.

We shall achieve~\eqref{eq: CLT goal} by splitting the integral into three
parts, and showing each could be made less than a given $\ep>0$ if $m \ge
\nu$ is chosen large enough.

Fix $R>0$ (to be determined later). By Taylor expansion,
\begin{equation}\label{eq: Taylor} G(x)=G(0)+x
G'(0)+\frac{x^2}{2}G''(0)+o(x^2) = 1 +\frac{\alpha x^2}{2}+o(x^2),\:
x\to 0
\end{equation} and so $|G(x/\sqrt{m})|^m \to e^{-\alpha
x^2/2}$ as $m\to \infty$, uniformly in $x\in[-R,R]$. Thus the integral
in~\eqref{eq: CLT goal} computed on $[-R,R]$ converges to zero as $m\to
\infty$.

From the expansion~\eqref{eq: Taylor} we get
$$\exists \delta>0 \ \forall |x|<\delta: \ |G(x)|\le e^{-\frac{\alpha x^2}{4}}. $$
Consider the integration in~\eqref{eq: CLT goal} for $R \le |x|\le
\delta\sqrt{m}$. For such $x$ we have $\left|G(x/\sqrt m )\right|^m \le
e^{-\frac {\alpha x^2}{4}}$, and so the integrand is less than
$2e^{-\frac{\alpha x^2}{4}}$. Choosing $R$ so that $4\int_{R}^\infty
e^{-\frac{\alpha x^2}{4}}<\ep$ will satisfy our needs.

Lastly, consider the integration on $\delta\sqrt{m}\le |x|<\infty$. By
properties of Fourier transform, $\eta:=\sup_{|x|\ge \delta} |G(x)| \in
(0,1)$. Thus

$$ \int_{|x|\ge \delta \sqrt{m}} \left|\  \left|G\left(\frac{x}
{\sqrt m}\right)\right|^m - e^{-\frac{\alpha x^2}{2}}\right|\ dx \le
\eta^{m-\nu}\ \sqrt{m}\int_{\R}|G|^\nu + \int_{|x|\ge \delta\sqrt{m}}e^{-\frac {\alpha x^2}{2}}dx < \ep,
$$
for $m$ large enough. Here we have used condition~\ref{item: G integ}.

\end{proof}

\end{document}